%% file: QMC_PDE_arxiv.tex
\documentclass[a4paper,11pt]{article}
\input{my_arxiv_header.tex}
\title{Quasi-Monte Carlo integration with product weights for elliptic PDEs with log-normal coefficients}
\usepackage[tikz]{mdframed}
\newcommand{\amin}{\check{a}}
\newcommand{\amax}{\hat{a}}
\newcommand{\mixedfirst}[2]{\frac{\partial^{|\fraku|} #1}{\partial #2_{\fraku}}} 
\newcommand{\Vdual}{{V^{'}}} 
\newcommand{\sumu}{\sum_{\fraku\subseteq\{1:s\}}}
\newcommand{\scbasis}{\psi}
\newcommand{\Besovp}{\mathtt{p}}
\newcommand{\Besovq}{\mathtt{q}}

\theoremstyle{definition}
\newtheorem{assumption}{Assumption}

\newtheorem{assumptionprime}{Assumption}

\bibliography{UNSW_1}

\newcommand{\Addresses}{{
	\bigskip
  \footnotesize
  \textsc{School of Mathematics and Statistics}\par\nopagebreak
  \textsc{University of New South Wales}\par\nopagebreak
  \textsc{Sydney NSW 2052, Australia}\vspace{1pt}\par\nopagebreak
  \textit{E-mail}: \texttt{\href{mailto:y.kazashi@unsw.edu.au}{y.kazashi@unsw.edu.au}}
}}

\author{Yoshihito Kazashi}
\begin{document}
\maketitle
\begin{abstract}
Quasi-Monte Carlo (QMC) integration {of} output functionals of solutions {of} the diffusion problem with a log-normal random coefficient is considered. 
The random coefficient is assumed to be given by {an} exponential of a Gaussian random field that is represented by {a series expansion of} some system of functions. 
Graham et al. \cite{GKNetal2014:lognormal} developed a lattice-based QMC theory for this problem and established {a} quadrature error decay {rate} $\approx1$ with respect to the number of quadrature points. The key assumption there was a suitable summability condition on the aforementioned system of functions. {As a consequence, product-order-dependent (POD) weights were used to construct the lattice rule. In this paper, a different assumption on the system is considered. This assumption, originally considered by Bachmayr et al. \cite{Bachmayr_etal_2016_ESAIM_part2} to utilise the locality of support of basis functions in the context of polynomial approximations applied to the same type of the diffusion problem, is shown to work well in the same lattice-based QMC method considered by Graham et al.: the assumption leads us to product weights, which enables the construction of the QMC method with a smaller computational cost than Graham et al.}
{A} quadrature error decay {rate} $\approx1$ is established, and {the theory developed here} is applied to a wavelet stochastic model. By {a} characterisation of the Besov smoothness, it is shown that a wide class of path smoothness can be treated with {this} framework.
\end{abstract}
\keywords{Quasi-Monte Carlo methods, Partial differential equations with random coefficients, Log-normal, Infinite dimensional integration}
\section{Introduction}
This paper is concerned with quasi-Monte Carlo (QMC) integration {of} output functionals of solutions {of} the diffusion problem with a random coefficient of the form
\begin{align}
-\nabla\cdot (a(x,\omega )\nabla u(x,\omega )) = f(x)\quad \text{in } D\subset \bbR^d,
\qquad 
u=0\quad\text{ on }\partial{\dom},
\label{eq:PDE}
\end{align}
where $\omega\in\Omega$ is an element of a suitable probability space $(\Omega,\scrF,\Prob)$ (clarified below), and $D\subset\bbR^d$ is a bounded domain with Lipschitz boundary. Our interest is in the log-normal case, that is, 
$a(\cdot,\cdot)\from {\dom}\times\Omega\to \bbR$ is assumed to have the form
\begin{align}
a(x,\omega )=a_*(x)+a_0(x)\exp(T(x,\omega ))\label{eq:random coeff}
\end{align}
with {continuous functions $a_*\ge0$, $a_0>0$,} and Gaussian random field $T(\cdot,\cdot)\from {D}\times\Omega\to \bbR$ represented by a series expansion
$$
T(x,\omega)=\sum_{j= 1}^\infty
Y_j (\omega)
\scbasis_{j}(x),\quad \text{ for all }x\in\dom,
$$ 
with a suitable system of functions $(\scbasis_{j})_{j\ge1}$.

To handle a wide class of $a$ and $f$, we consider the weak formulation of the problem \eqref{eq:PDE}. By $V$ we denote the zero-trace Sobolev space $H^1_0(\dom)$ endowed with the norm
\begin{align}
\norm{v}_V := \Big( \int_{\dom} |\nabla v(x)|^2\dx\Big)^{\frac12},
\end{align}
and by $V':=H^{-1}(\dom)$ the topological dual space of $V$.
For the given random coefficient $a(x,\omega )$, we define the {bilinear} form  $\mathscr{A}(\omega;\cdot,\cdot)\from V\times V\to\bbR$ by
\begin{align}
\Omega\ni\omega 
\mapsto
\mathscr{A}(\omega;v,w)
:=
\int_{\dom} a(x,\omega ) \nabla v(x)\cdot \nabla w(x)\dx\ 
\text{ for all }v,w\in V.
\end{align}
Then, for any $\omega\in\Omega$, the weak formulation of \eqref{eq:PDE} reads:
 {f}ind $u(\cdot,\omega)\in V$ such that 
\begin{align}
\mathscr{A}(\omega;u(\cdot,\omega),v)
=\langle f, v \rangle\quad
\text{ for all }v\in V,
\label{eq:wk formulation omega}
\end{align}
where $f$ is assumed to be in $V'$, and $\langle \cdot, \cdot \rangle$ denotes the duality paring between $V'$ and $V$. We impose further {conditions} to ensure the well-posedness of the problem, which we will discuss later. 

The ultimate goal is to compute $\E[\G(u(\cdot))]$, the expected value of $\G(u(\cdot,\omega))$, where $\G$ is a linear bounded functional on $V$. The problem \eqref{eq:PDE}, and of computing $\E[\G(u(\cdot))]$ often arises in many applications such as hydrology \cite{Dagan.G_1984_solute_transport,Naff.R.L_et_al_1998_part1,Naff.R.L_et_al_1998_part2}, and has attracted attention in computational uncertainty quantification (UQ). See, for example, \cite{Cohen.A_DeVore_2015_Acta,Schwab.C_Gittleson_2011_Acta,Kuo.F_Nuyens_2016_FoCM_survey} and references therein. Two major ways to tackle this problem are function approximation, and quadrature, in particular, quasi-Monte Carlo (QMC) methods.

Our interest is in QMC. It is now well known that the QMC methods beats the plain-vanilla Monte Carlo methods in various {settings} when applied to the problems of computing $\E[\G(u(\cdot))]$ (\cite{GKNetal2014:lognormal,Kuo.F_Nuyens_2016_FoCM_survey,Kuo.F_Schwab_Sloan_2012_SINUM}). Among the QMC methods, the algorithm we consider is \textit{randomly shifted lattice rules}.

Graham et al. \cite{GKNetal2014:lognormal} showed {that} when the randomly shifted lattice rules {are} applied to the class of PDEs we consider, a QMC convergence rate, in terms of expected root square mean root, $\approx1$ is achievable, which is known to be optimal for lattice rules in the function space they consider. More precisely, they showed {that} quadrature points for randomly shifted lattice rules that achieve such a rate can be constructed using an algorithm called component-by-component (CBC) construction. 
The algorithm {uses} \textit{weights}, which represents the relative importance of {subsets of} the variables of the integrand, as an input, and the cost of it is dependent on the type of weights. The weights considered in \cite{GKNetal2014:lognormal} are so-called product-order-dependent (POD) weights, {which were determined by minimising an error bound.} {For POD weights, }the CBC construction takes $\mathcal{O}(s n\log n+s^2n)$ operations, where $n$ is the number of quadrature points and $s$ is the dimension of truncation $\sum_{j= 1}^{s}
Y_j (\omega)
\scbasis_{j}(x)$. 

The contributions of the current paper are twofold: proof of a convergence rate $\approx1$ with product weights, and an application to a stochastic model with wavelets. 
In more detail, we show that for the currently considered problem, {the} CBC construction can be constructed with weights called product weights, and achieves the optimal rate $\approx1$ {in the function space we consider}, and further, we show that the developed theory can be applied to a stochastic model {which covers} a wide class of wavelet {bases}. 

Often in practice, we want to approximate the random coefficients well, and consequently $s$ {has to be} taken to be large, in which case the second term of $\mathcal{O}(s n\log n+s^2n)$ becomes dominant. 
{The use of the POD weights originates from the summability condition imposed on $(\psi_j)$ by Graham et al. \cite{GKNetal2014:lognormal}.
We consider a different condition, the one proposed by Bachmayr et al. \cite{Bachmayr_etal_2016_ESAIM_part2} to utilise the locality of supports of $(\psi_j)$ in the context of polynomial approximations applied to PDEs with random coefficients. We show that under this condition,} the shifted lattice rule for the PDE problem can be constructed with {a} CBC algorithm with the {computational cost} $\mathcal{O}(s n\log n)$, the cost with the product weights as shown in \cite{Dick.J_etal_2014_higher_order_Galerkin}. Further, the stochastic model we consider broadens the range of applicability of the QMC methods to the PDEs with log-normal coefficients. 
One concern about the conditions, in particular the summability condition on $(\scbasis_j)$, imposed in \cite{GKNetal2014:lognormal} is that it is so strong that only random coefficients with smooth realisations are in the scope of the theory. We show that at least for $d=1,2$, such random coefficients (e.g., realisations {with} just some H\"{o}lder smoothness) can be considered.

We note that the similar argument employed in the current paper is applicable to the randomly shifted lattice rules applied to PDEs with uniform random coefficients considered in \cite{Kuo.F_Schwab_Sloan_2012_SINUM}. One of the keys in the current paper is the estimate of the derivative given in Corollary \ref{cor:pd uy lognormal}. This result essentially follows from the results by Bachmayr et al. \cite{Bachmayr_etal_2016_ESAIM_part2}. The paper \cite{Bachmayr_etal_2016_ESAIM_part1}, Part I of their work \cite{Bachmayr_etal_2016_ESAIM_part2}, considers the uniform case, and the similar argument as the one presented here turns out to work almost in parallel.

Upon finalising this paper, we learnt about the two papers, one by Gantner et al. 
\cite{Gantner.R.N_etal_2016_SAM_report}, and the other by Herrmann and Schwab \cite{Herrmann.L_Schwab_2016_SAM_report}. 
Our works share the same spirit in that we {are all} inspired by the work by Bachmayr et al. \cite{Bachmayr_etal_2016_ESAIM_part1,Bachmayr_etal_2016_ESAIM_part2}. 
The interest of Gantner et al. \cite{Gantner.R.N_etal_2016_SAM_report} is in the uniform case. They consider not only the randomly shifted lattice rules but also higher order QMCs. Since our interest was on the randomly shifted lattice rule in the uniform case, and our results are a proper subset of their work \cite{Gantner.R.N_etal_2016_SAM_report}, we {defer} to \cite{Gantner.R.N_etal_2016_SAM_report} for the uniform case. 

As for the log-normal case we provide a different, {arguably simpler, }proof for the same convergence rate with the exponential weight function, and we discuss the roughness of the realisations that can be considered.

Herrmann and Schwab \cite{Herrmann.L_Schwab_2016_SAM_report} develops a theory under the setting essentially the same as ours. In contrast to our paper, they treat the truncation error in a general setting, and as for the QMC integration error, they consider both the exponential weight functions and the Gaussian weight function for the weighted Sobolev space.
As for the exponential weight function, the current paper and \cite{Herrmann.L_Schwab_2016_SAM_report} impose essentially the same assumptions 
(Assumption \ref{assump:B} below),
 and show the same convergence rate. However, our proof strategy is different, which turns out to result in different (product) weights, (and a different constant, although it does not seem to be easy to say which is bigger). 
Further, in contrast to \cite{Herrmann.L_Schwab_2016_SAM_report}, we provide a discussion of the roughness of the realisations of random coefficients as mentioned above. The log-normal case, in comparison to the uniform case where the ``random parameter{s}'' can be uniformly bounded, is ``intrinsically random'' in the sense that the {magnitude of each} parameter can be {arbitrarily large}. As a consequence, the connection between the smoothness of the spatial basis and the one of the smooth realisations are not immediately clear. 
In Section \ref{sec:applications}, we provide a discussion via the Besov \emph{characterisation} of the realisations of the random coefficients and the embedding results. 

The outline of the rest of the paper is as follows. 
In Section \ref{sec:setting}, we describe the problem we consider in detail. Then, in Section 
\ref{sec:QMC error} we  develop the QMC theory applied to the PDE problem with log-normal coefficients using the product weights.
Section \ref{sec:applications} provides an application of the theory{:} we consider a stochastic model represented by {a} wavelet Riesz basis. 
Then, we close this paper with concluding remarks in Section \ref{sec:conclusion}.

\section{Setting}\label{sec:setting}
We assume {that} the Gaussian random field $T$ admits a series representation
$
T(x,\omega)=\sum_{j= 1}^\infty
Y_{j}(\omega)
\scbasis_{j}(x),
$
where $\{Y_j\}$ is a collection of independent standard normal random variables on a suitable probability space $(\Omega,\scrF,\Prob)$, and $(\scbasis_{j})$ is a system of real-valued measurable functions on $\dom$. 
For simplicity we fix $(\Omega,\scrF,\Prob):=(\bbR^{\bbN},\mathcal{B}(\bbR^{\bbN}),\Prob_{Y})$, where $\bbN:=\{1,2,\dotsc,\}$, $\mathcal{B}(\bbR^{\bbN})$ is the Borel $\sigma$-algebra generated by the product topology in $\bbR^{\bbN}$, and $\Prob_{Y}:=\prod_{j=1}^{\infty}\Prob_{Y_j}$ is the product measure on $(\bbR^{\bbN},\mathcal{B}(\bbR^{\bbN}))$ defined by the standard normal distributions $\{\Prob_{Y_j} \}_{j\in\bbN}$ on $\bbR$ (see, for example, \cite[Chapter 2]{Ito.K_1984_book_intro_prob} for details). 
Then, for each $\bs{y}\in \Omega$ we may see $Y_j(\bs{y})$ ($j\in\bbN$) as given by the projection (or the canonical coordinate function)
$$
\Omega=\bbR^{\bbN}\ni
\bs{y}\mapsto Y_j(\bs{y})=:y_j\in\bbR.
$$
Note in particular that from the continuity of the projection, the mapping
$\bs{y}\mapsto y_j$ is $\mathcal{B}(\bbR^{\bbN})/\mathcal{B}(\bbR)$-measurable.

In the following, we write $T$ above as
\begin{align}
T(x,\bs{y})=\sum_{j= 1}^\infty
y_j 
\scbasis_{j}(x),\quad \text{ for all }x\in\dom,
\label{eq:rf rep}
\end{align}
and see it as a deterministically parametrised function on $\dom$. 
We will impose a condition considered by Bachmayr et al. \cite{Bachmayr_etal_2016_ESAIM_part2} on $(\psi_j)$, see Assumption \ref{assump:B} {below}, that {is} particularly suitable for $\psi_j$ with local support.
 
To ensure the law on $\bbR^\dom$ is well defined, we suppose 
\begin{align}
\sum_{j= 1}^\infty
\scbasis_{j}(x)^2<\infty\text{ for all }x\in\dom,
\label{eq:cov finite}
\end{align}
so that the covariance function $\E[T(x_1)T(x_2)]=\sum_{j\ge1}\scbasis_{j}(x_1)\scbasis_{j}(x_2)$ ($x_1,x_2\in\dom$) is well-defined.
We consider the parametrised elliptic partial differential equation
\begin{align}
-\div(a(x,\bs{y})\nabla u(x,\bs{y}))=f(x)\text{ in }\dom,\quad
u=0\text{ on }\partial\dom,
\label{eq:param PDE}
\end{align}
where
\begin{align}
a(x,\bs{y})=a_*(x)
+
a_0(x)
\exp
\big(T(x,\bs{y})\big),
\end{align}
with continuous functions $a_*$, $a_0$ on $\bddom$. We assume $a_*$ is non-negative on $\bddom$, and $a_0$ is positive on $\bddom$.

In accordance with the above formulation, we rewrite \eqref{eq:wk formulation omega} as the parametrised variational problem: find $u\in V$ such that
\begin{align}
\mathscr{A}(\bs{y};u(\cdot,\bs{y}),v)
=\langle f, v \rangle\quad
\text{ for all }v\in V,
\label{eq:wk formulation param}
\end{align}
{To} prove well-posedness of the variational problem \eqref{eq:wk formulation param}, we use the Lax--Milgram lemma. Conditions which ensure that the bilinear form $\mathscr{A}(\bs{y};\cdot,\cdot)$ defined by the diffusion coefficient $a$ is coercive and bounded are discussed later.

Motivated by UQ applications, we are interested in expected values of bounded linear functionals of the solution of the above PDEs. That is, given a continuous linear functional $\G\in\Vdual$ we wish to compute $\E[\G(u(\cdot))]:=\int_{\bbR^\bbN}\G(u(\cdot,\bs{y}))\mathrm{d}\Prob_{Y}(\bs{y})$, where the measurability of the integrands will be discussed later. 
To compute $\E[\G(u(\cdot))]$ we use a sampling method: generate realisations of $a(x,\bs{y})$, which yields the solution $u(x,\bs{y})$ via the PDE \eqref{eq:param PDE}, and from these we compute $\E[\G(u(\cdot))]$.

In practice, these operations cannot be performed exactly, and numerical methods {need to be} employed. This paper gives an analysis of the error incurred by the method outlined as follows.  
We compute the realisations by truncation, that is, for some integer $s\ge1$ we generate $a(x,(y_1,\dotsc,y_s,0,0,0,\dots))$. 
Further, the expectation is approximated by a QMC method. 

Let $u^s(x)=u^s(x,\bs{y})$ be the solution of \eqref{eq:param PDE} with $\bs{y}=(y_1,\dotsc,y_s,0,0,0,\dotsc)$, that is, {of} the problem: find $u^s\in V$ such that 
\begin{align}
-\div(a(x,(y_1,\dotsc,y_s,0,0,\cdots))\nabla u^s(x) = f(x) \ \ \text{in } \dom,
\quad             u^s = 0  \text{ on } \partial\dom.
\end{align}
Here, even though the dependence of $u^s$ on $\bs{y}$ is only on $(y_1,\dotsc,y_s)$, we abuse the notation slightly by writing
 $u^s(x,\bs{y}):=
 u^s(y_1,\dotsc,y_s,0,0,0,\dotsc)$.

Let $\Phi_s\inv\from[0,1]^s\ni\bs{v}\mapsto\Phi_s\inv(\bs{v})\in \bbR^s$ be the inverse of the {cumulative} normal distribution function applied to each entry of $\bs{v}$. We write $F(\bs{y}):=F(y_1,\dotsc,y_s)=\G(u^s(\cdot,\bs{y}))$ and
\begin{align}
I_s(F)
&:=
\int_{{\bs{v}\in(0,1)^s}}
F(\Phi_s\inv(\bs{v}))
\mathrm{d}{\bs{v}}
=\int_{{\bs{y}\in\mathbb{R}^s}}
\G(u^s(\cdot,\bs{y}))
\prod\limits_{j=1}^{s}\phi(y_j)\mathrm{d}\bs{y}
=\E[\G(u^s)],
\end{align}
where $\phi$ is the probability density function of the standard normal random variable. The measurability of the mapping $\bbR^s\ni\bs{y}\mapsto \G(u^s(\cdot,\bs{y}))\in \bbR$ will be discussed later.

In order to approximate $I_s(F)$, we employ a QMC method called a \textit{randomly shifted lattice rule}. This is an equal-weight quadrature rule of the form
$$
\mathcal{Q}_{s,n}(\bs{\Delta};F)
:=\frac{1}n\sum_{i=1}^nF\left(\Phi\inv_s\left(\mathrm{frac}\left(\frac{i\bs{z}}{n}+\bs{\Delta}\right)\right)\right),
$$
where the function $\mathrm{frac}(\cdot)\from\bbR^s\ni\bs{y}\mapsto\mathrm{frac}(\bs{y})\in[0,1)^s$ takes the fractional part of each component in $\bs{y}$. 
Here, $\bs{z}\in\bbN^s$ is a carefully chosen point called the (deterministic) \textit{generating vector} and $\bs{\Delta}\in[0,1]^s$ is the \textit{random shift}. We assume the random shift $\bs{\Delta}$ is a $[0,1]^s$-valued uniform random variable defined on a suitable probability space different from $(\Omega,\mathscr{F},\Prob)$. For further details of the randomly shifted lattice rules, we refer to the surveys \cite{Dick.J_etal_2013_Acta,Kuo.F_Nuyens_2016_FoCM_survey} and references therein.

We want to evaluate the root-mean-square error
\begin{align}
\sqrt{
\E^{\bs{\Delta}}
\left[
\big(
\E[\G(u)] -\mathcal{Q}_{s,n}(\bs{\Delta};F)
\big)^2
\right]
}.
\end{align}
where $\E^{\bs{\Delta}}$ is the expectation with respect to the random shift. 
Note that in practice the solution $u^s$ needs to be approximated by some numerical scheme $\widetilde{u}^s$, which results in computing  $\widetilde{F}(\bs{y}):=\G(\widetilde{u}^s(\bs{y}))$. Thus, the error
$e_{s,n}:=\sqrt{
\E^{\bs{\Delta}}
\left[
\big(
\E[\G(u)] -\mathcal{Q}_{s,n}(\bs{\Delta};\widetilde{F})
\big)^2
\right]
}$ is what we need to evaluate in practice. 
Via the trivial decompositions we have, using 
$\E^{\bs{\Delta}}[\mathcal{Q}_{s,n}(\bs{\Delta};\widetilde{F})]=\E[\G(\widetilde{u})]$ (see, for example, \cite{Dick.J_etal_2013_Acta}),
\begin{align}
e_{s,n}^2
&=
(\E[\G(u)-\G(\widetilde{u}^s)])^2
+
\E^{\bs{\Delta}}
\left[
\big(
\E[\G(\widetilde{u}^s)] -\mathcal{Q}_{s,n}(\bs{\Delta};\widetilde{F})
\big)^2
\right] \\
&\le 
2(\E[\G(u)-\G(\widetilde{u})])^2
+
2(\E[\G(\widetilde{u})-\G(\widetilde{u}^s)])^2 
+\E^{\bs{\Delta}}
\left[
\big(
\E[\G(\widetilde{u}^s)] -\mathcal{Q}_{s,n}(\bs{\Delta};\widetilde{F})
\big)^2
\right],
\end{align}
where $\widetilde{u}$ is an approximation of the solution $u$ of \eqref{eq:param PDE} with the same scheme as $\widetilde{u}^s$.

 For the sake of simplicity, we forgo the discussion on the numerical approximation of the solution of the PDE. 
 Instead, we discuss the smoothness of the realisations of the random coefficient. Then, given a suitable smoothness of the boundary $\partial\dom$, the convergence rate of $\E[\G(u)-\G(\widetilde{u}^s)]$ is typically obtained from the smoothness of the realisations of the coefficients $a(\cdot,\bs{y})$, via the regularity of the solution $u$. See 
 \cite{GKNetal2014:lognormal,Kuo.F_Nuyens_2016_FoCM_survey,Kuo.F_Schwab_Sloan_2012_SINUM}. Therefore in the following, we concentrate on the truncation error and the quadrature error, the second and the third term of the above decomposition, and the realisations of $a$.

In the course of the error analyses, we assume $(\scbasis_j)$ satisfies the following assumption.
\begin{assumption}\label{assump:B}
The system $(\scbasis_j)$ satisfies the following. There exists a positive sequence $(\rho_j)$ such that 
\begin{align}
\sup_{x\in {\dom} }\sum_{j\ge1}\rho_j|\psi_j(x)|=:\kappa <\ln2,
\tag{\textbf{b1}}
\label{eq:cond psi b1}
\end{align}
and further,
\begin{align}
(1/\rho_j)\in\ell^{q}\qquad\text{ for some }{q}\in(0,1].
\tag{\textbf{b2}}\label{eq:cond 1/rho b2}
\end{align}
\end{assumption}
We also use the following weaker assumption.
\begin{assumptionprime}\label{assump:Bprime}
The same as Assumption \ref{assump:B}, only {with} the condition \eqref{eq:cond 1/rho b2} being replaced with 
\begin{align}
(1/\rho_j)\in\ell^{q}\qquad\text{ for some }q\in(0,\infty).
\tag{\textbf{b2\ensuremath{^\prime}}}
\label{eq:cond 1/rho b2prime}
\end{align}
\end{assumptionprime}
{We note that \eqref{eq:cond 1/rho b2prime}, and thus also \eqref{eq:cond 1/rho b2}, implies $\rho_j\to\infty$ as $j\to\infty$.}

Some remarks on the assumptions are in order. First note that Assumption \ref{assump:Bprime} implies $\sum_{j\ge1}|\psi(x)|<\infty$ for any $x\in\dom$, {and hence} \eqref{eq:cov finite}. Assumption \ref{assump:Bprime} is used to obtain an estimate on the mixed derivative with respect to the random parameter $y_j$, and further, ensures the almost surely well-posedness of the problem \eqref{eq:wk formulation param} --- see Corollary \ref{cor:pd uy lognormal} and Remark \ref{rem:on bbdness}. Assumption \ref{assump:B} is used to obtain a {dimension-independent} QMC error estimate --- see Theorem \ref{thm:conv rate}, and Theorem \ref{thm:wavelet QMC error}. The stronger the condition \eqref{eq:cond 1/rho b2} the system $(\psi_j)$ satisfies, {that is, the smaller is $q$}, the smoother the realisations of the random coefficient become. In Section \ref{sec:Holder smoothness}, we discuss smoothness of realisations allowed by these conditions.
\section{Bounds on mixed derivatives}
In this section, we discuss bounds on mixed derivatives. In order to motivate the discussion in this section, first we explain how the derivative bounds come into play in the QMC analysis developed in the next section.

Application of QMC methods to elliptic PDEs with log-normal random coefficients was initiated with computational results by Graham et al. \cite{Graham.I_etal_2011_JCP}, and an analysis was followed {by} Graham et al. \cite{GKNetal2014:lognormal}. Following the discussion by \cite{GKNetal2014:lognormal}, we assume the integrand $F$ is in the space called {the} \textit{weighted unanchored Sobolev space} $\mathcal{W}^s$, consisting of measurable functions $F\from \bbR^s\to\bbR$ such that
\begin{align}
\norm{F}_{\mathcal{W}^s}^2
=\!\!
\sumu
\frac1{\gamma_{\fraku}}
\int_{\mathbb{R}^{|\fraku|}}\!\!\!
\left(
\int_{\mathbb{R}^{s-|\fraku|}}
\mixedfirst{F}{y}(\bs{y}_\fraku;\bs{y}_{\{1:s\}\setminus\fraku})
\prod\limits_{j\in\{1:s\}\setminus\fraku}
\phi(y_j)\dbsy_{\{1:s\}\setminus\fraku}
\right)^2\!\!
\prod\limits_{j\in\fraku}
w_j^2(y_j)\dbsy_\fraku<\infty,
\label{eq:F norm}
\end{align}
where we assume, similarly to \cite{GKNetal2014:lognormal}, {that}
\begin{align}
w_j^2(y_j)=\exp(-2\alpha_j|y|)\label{eq:weight fct}
\end{align}
for some $\alpha_j>0$. 
Here, $\{1:s\}$ is a shorthand notation for the set $\{1,\dotsc,s\}$, $\mixedfirst{F}{y}$ denotes the mixed first derivative with respect to each of the ``active'' variables $y_j$ with $j\in\fraku\subseteq\{1:s\}$, and $\bs{y}_{\{1:s\}\setminus\fraku}$ denotes the ``inactive'' variables $y_j$ with $j\not\in\fraku$. Further, weights $({\gamma_{\fraku}})$ {describe} the relative importance of the variables $\{y_j\}_{j\in\fraku}$. Note that the measure $\int_{\cdot}\dy_{\fraku}$ and $\int_{\cdot}\frac1{\gamma_{\fraku}}\dy_{\fraku}$ differ by at most a constant factor {depending on $\fraku$}. Weights $({\gamma_{\fraku}})$ play an important role {in deriving} error estimates independently of the dimension $s$, and further, {in obtaining} the generating vector $\bs{z}$ for the lattice rule via {the {component-by-component (CBC)} algorithm.} 

Depending on the problem, different types of weights have been considered to derive error estimates. For the randomly shifted lattice rules, ``POD weights'' and ``product-weights'' {have been} considered (\cite{Dick.J_etal_2013_Acta,Kuo.F_Nuyens_2016_FoCM_survey}). When applied to the PDE parametrised with log-normal coefficients, the result in \cite{GKNetal2014:lognormal} suggests the use of POD weights for the problem.

We wish to develop a theory on the applicability of product weights, which has an advantage in terms of computational cost. The computational cost of the CBC construction is $\mathcal{O}(sn\log n+ns^2)$ in the case of POD weights, compared to $\mathcal{O}(sn\log n)$ for product weights \cite{Dick.J_etal_2014_higher_order_Galerkin}. Since we often want to approximate the random field well, {and so necessarily} we have large $s$, the applicability of product weights is of clear interest. 

Estimates {of} derivatives of the integrand $F(\bs{y})$ with respect to the parameter $\bs{y}$, that is, the variable with which $F(\bs{y})$ is integrated, are one of the keys in the error analysis of QMC. 
In \cite{GKNetal2014:lognormal}, it was the estimates being {of} ``POD-form'' that led their theory to the POD weights. 
Under an assumption on the system $(\scbasis_j)${, which is} different from {that in} \cite{GKNetal2014:lognormal}, we show that the derivative estimates turn out to be {of} ``product-form'', and further that, under a suitable assumption, we achieve the same error convergence rate {close to $1$} with product weights. 
%

Now, we derive an estimate of the product form. Let $\mathcal{F}:=\{\mu=(\mu_1,\mu_2,\dotsc)\in\mathbb{N}_0^{\mathbb{N}}\mid \text{ all but finite number of components of }\mu\text{ are zero} \}$.
For $\mu\in\mathcal{F}$ we use the notation 
$|\mu|=\sum_{j\ge1}\mu_j$, 
${\mu}!=\prod\limits_{j\ge1}\mu_j!$, 
$\rho^{\mu}=\prod\limits_{j\ge1}\rho_j^{\mu_j}$ 
{for $\rho=(\rho_j)_{j\ge1}\in \mathbb{R}^{\mathbb{N}}$}, and 
\begin{align}
\partial^\mu u= \frac{\partial^{|\mu|}}{y_{j(1)}^{\mu_{j(1)}}\dotsb y_{j{(k)}}^{\mu_{j(k)}}}u,
\end{align}
where $k=\#\{j\mid\mu_j\neq 0\}$.

We have the following bound {on} mixed derivatives {of order $r\ge1$ (although in our application we will need only $r=1$)}. 
The proof follows essentially the same argument as the proof by Bachmayr et al. \cite[Theorem 4.1]{Bachmayr_etal_2016_ESAIM_part2}. Here, we show a tighter bound by changing the condition from $\frac{\ln2}{\sqrt{r}}$ to $\frac{\ln2}{{r}}$ in 
{\cite[(91)]{Bachmayr_etal_2016_ESAIM_part2}}, and we have $\rho^{2\mu}$ {in \eqref{eq:Bachmayr etal'}} in place of $\frac{\rho^{2\mu}}{\mu!}$ in the left hand side of \cite[(92)]{Bachmayr_etal_2016_ESAIM_part2}.
\begin{prop}\label{prop:pd uy lognormal}
Let $r\ge1$ be an integer. 
Suppose $(\psi_j)$ satisfies the condition \eqref{eq:cond psi b1} 
{with $\ln 2$ replaced by $\frac{\ln 2}{r}$,}
with a positive sequence $(\rho_j)$. 
Then, there exists a constant $C_0=C_0(r)$ that depends on $\kappa$ and $r$, such that
\begin{align}
\sum_{\substack{\mu\in\mathcal{F}\\\norm{\mu}_\infty\le r}}
{\rho^{2\mu}}
\int_{\dom} a(\bs{y})|\nabla(\partial^\mu u(\bs{y}))|^2\dx
\le
C_0 
\int_{\dom} a(\bs{y})|\nabla u(\bs{y})|^2 \dx.
\label{eq:Bachmayr etal'}
\end{align}
 for all $\bs{y}$ that satisfies $\big\|\sum_{j\ge1}y_j\psi_j\big\|_{L^{\infty}(\dom)}<\infty$, 
 where $u(\bs{y})$ is the solution of \eqref{eq:wk formulation param} for such $\bs{y}$. 
The same bound holds also for $u^s(\bs{y})$, the solution of $\eqref{eq:wk formulation param}$ with $\bs{y}=(y_1,\dotsc,y_s,0,0,\dotsc)$.
\end{prop}
\begin{proof}
Let 
\begin{align*}
\Lambda_{k}:=
&\{\mu\in\mathcal{F}\mid|\mu|=k\text{ and }\norm{\mu}_{\ell_\infty}\le r\},\ \, \text{and}\ \,
S_{\mu}:=
\{\nu\in\mathcal{F}\mid\nu\le \mu\text{ and }\nu\neq\mu\}\text{ for }\mu\in\mathcal{F},
\end{align*}
{with $\le$ denoting the component-wise partial order between multi-indices.}
Let us introduce the notation $
\norm{v}_{a(\bs{y})}^2:=\int_{\dom} a(\bs{y})|\nabla v|^2 \dx
$ for all $v\in V$, and let $$\sigma_k:=\sum_{\mu\in \Lambda_k}\rho^{2\mu}
\norm{\partial^\mu u(\bs{y})}_{a(\bs{y})}^2.$$
 We show below that we can choose $\delta=\delta(r)<1$ such that
\begin{align}
\sigma_k\le \sigma_0\delta^k\qquad\text{ for all }k\ge0.
\end{align}
Note that if this holds then we have
\begin{align}
\sum_{\norm{\mu}_\infty\le r}{\rho^{2\mu}}\norm{\partial^\mu u(\bs{y})}_{a(\bs{y})}^2
=
\sum_{k=0}^\infty \sum_{\mu\in\Lambda_k}\rho^{2\mu}\norm{\partial^\mu u(\bs{y})}_{a(\bs{y})}^2=
\sum_{k=0}^\infty \sigma_k\le \sigma_0\sum_{k=0}^\infty \delta^k<\infty,
\end{align}
and the statement will follow with $C_0=C_0(r)=\sum_{k=0}^\infty \delta(r)^k$.

We now show $\sigma_k\le \sigma_0\delta^k$. Note that from the assumption  $\|\sum_{j\ge1}y_j\psi_j\|_{L^{\infty}(\dom)}<\infty$, in view of \cite[Lemma 3.2]{Bachmayr_etal_2016_ESAIM_part2} we have $\partial^\mu u\in V$ for any $\mu\in\mathcal{F}$. Thus, by taking $v:=\partial^\mu u$ ($\mu\in \Lambda_k$) in \cite[(74)]{Bachmayr_etal_2016_ESAIM_part2}, we have
\begin{align}
{\sigma_k}=&\sum_{\mu\in \Lambda_k}\rho^{2\mu}\!\!\!
\int_D\!\!
a(\bs{y})|\nabla\partial^\mu u(\bs{y})|^2\dx\notag\\
\le&
\sum_{\mu\in \Lambda_k}
\sum_{\nu\in S_\mu}
\Bigg(\prod\limits_{j\ge 1}
\frac{\mu_j!\rho_j^{\mu_j-\nu_j}\rho^{\mu_j}\rho^{\nu_j}}
{\nu_j!(\mu_j-\nu_j)!}
\Bigg)
\int_D
a(\bs{y})
\Bigg(
\prod\limits_{j\ge 1}|\psi_j|^{\mu_j-\nu_j}
\Bigg)
|\nabla\partial^\nu u(\bs{y})|
|\nabla\partial^\mu u(\bs{y})|
\dx.
\end{align}
Using the notation
\begin{align}
\epsilon(\mu,\nu)(x)
:=
\epsilon(\mu,\nu)
:=
\frac{\mu!}{\nu!}\frac{\rho^{\mu-\nu}|\psi|^{\mu-\nu}}{(\mu-\nu)!},
\end{align}
and the Cauchy--Schwarz inequality for the sum over $S_\mu$, it follows that
\begin{align}
{\sigma_k}&
\le
\int_D
\sum_{\mu\in \Lambda_k}
\sum_{\nu\in S_\mu}
\epsilon(\mu,\nu)
a(\bs{y})
|\rho^\nu\nabla\partial^\nu u(\bs{y})|
|\rho^\mu\nabla\partial^\mu u(\bs{y})|
\dx\\
&\le
\int_D
\sum_{\mu\in \Lambda_k}
\left(
\sum_{\nu\in S_\mu}
\epsilon(\mu,\nu)
a(\bs{y})
|\rho^\nu\nabla\partial^\nu u(\bs{y})|^2
\right)^\frac{1}{2}
\left(
\sum_{\nu\in S_\mu}
\epsilon(\mu,\nu)
a(\bs{y})
|\rho^\mu\nabla\partial^\mu u(\bs{y})|^2
\right)^\frac{1}{2}
\dx.\label{eq:C--S sum}
\end{align}
Let $$S_{\mu,\ell}:=\{\nu\in S_\mu\mid |\mu-\nu|=\ell\}.$$ 
Then, for $\mu\in\Lambda_k$ we have 
$$S_{\mu}=\{\nu\in\mathcal{F}\mid\nu\le \mu,\ \nu\neq\mu\}=\bigcup\limits_{\ell=1}^{|\mu|}\{\nu\in\mathcal{F}\mid\nu\le \mu,\ |\mu-\nu|=\ell \}
=
\bigcup\limits_{\ell=1}^{|\mu|}S_{\mu,\ell}
,$$
{and further, from} $|\mu|=k$, we have
\begin{align}
\sum_{\nu\in S_\mu}
\epsilon(\mu,\nu)
=
\sum_{\ell=1}^{k}
\sum_{\nu\in S_{\mu,\ell}}
\epsilon(\mu,\nu)
=
\sum_{\ell=1}^{k}
\sum_{\nu\in S_{\mu,\ell}}
\frac{\mu!}{\nu!}\frac{\rho^{\mu-\nu}|\psi|^{\mu-\nu}}{(\mu-\nu)!}.
\label{eq:sum ep mu nu}
\end{align}
Since $\nu\in S_{\mu,\ell}$ implies $\sum_{j\in\supp\mu}(\mu_j-\nu_j)=\ell$, there are $\ell$ factors in $\frac{\mu!}{\nu!}=\prod_{j\in\supp\mu}\mu_j(\mu_j-1)\dotsb(\nu_j+1)$.
From $\mu_j\le r$ ($j\in\supp\mu$), each of the factors is at most $r$. Thus,
$$
\frac{\mu!}{\nu!}
\le r^\ell\quad\text{ for }\mu\in \Lambda_k,\ \nu\in S_{\mu,\ell}.
$$
Therefore, from the multinomial theorem, for each $x\in\dom$ it follows from 
\eqref{eq:sum ep mu nu} that
\begin{align}
\sum_{\nu\in S_\mu}
\epsilon(\mu,\nu)
\le&
\sum_{\ell=1}^{k}
r^\ell
\sum_{\nu\in S_{\mu,\ell}}
\frac{\rho^{\mu-\nu}|\psi|^{\mu-\nu}}{(\mu-\nu)!}
\le
\sum_{\ell=1}^{k}
r^\ell
\sum_{|\tau|=\ell}
\frac{\rho^{\tau}|\psi|^{\tau}}{\tau!}
=
\sum_{\ell=1}^{k}
r^\ell
\frac{1}{\ell!}
\sum_{|\tau|=\ell}
\frac{\ell!}{\tau!}
\rho^{\tau}|\psi|^{\tau}\\
=&
\sum_{\ell=1}^{k}
r^\ell
\frac{1}{\ell!}
(\sum_{j=1}^\infty \rho_j|\psi_j|)^\ell
\le
\sum_{\ell=1}^{k}
r^\ell
\frac{1}{\ell!}
\kappa^\ell
\le 
e^{r\kappa}-1\le e^{{\ln 2}}-1=1.
\end{align}
Inserting into \eqref{eq:C--S sum}, we have
\begin{align}
\sum_{\mu\in \Lambda_k}\!\!\rho^{2\mu}
\norm{\partial^\mu u(\bs{y})}_{a(\bs{y})}^2
\le
\int_D
\sum_{\mu\in \Lambda_k}
\left(
\sum_{\nu\in S_\mu}
\epsilon(\mu,\nu)
a(\bs{y})
|\rho^\nu\nabla\partial^\nu u(\bs{y})|^2
\right)^\frac{1}{2}
\left(
a(\bs{y})
|\rho^\mu\nabla\partial^\mu u(\bs{y})|^2
\right)^\frac{1}{2}
\dx.
\end{align}
Again {applying} the Cauchy--Schwarz inequality to the summation over $\Lambda_k$ and then to the integral, we have
\begin{align*}
\begin{split}
\sigma_k
&\le
\int_D
\left(
\sum_{\mu\in \Lambda_k}
\sum_{\nu\in S_\mu}
\epsilon(\mu,\nu)
a(\bs{y})
|\rho^\nu\nabla\partial^\nu u(\bs{y})|^2
\right)^\frac{1}{2}\!\!
\left(
\sum_{\mu\in \Lambda_k}
a(\bs{y})
|\rho^\mu\nabla\partial^\mu u(\bs{y})|^2
\right)^\frac{1}{2}
\dx\notag
\end{split}
\\
&\le 
\left(\int_D
\sum_{\mu\in \Lambda_k}
\sum_{\nu\in S_\mu}
\epsilon(\mu,\nu)
a(\bs{y})
|\rho^\nu\nabla\partial^\nu u(\bs{y})|^2
\dx
\right)^\frac{1}{2}
\sigma_k^{\frac12},
\end{align*}
{and hence}
\begin{align}
{\sigma_k\le
\int_D
\sum_{\mu\in \Lambda_k}
\sum_{\nu\in S_\mu}
\epsilon(\mu,\nu)
a(\bs{y})
|\rho^\nu\nabla\partial^\nu u(\bs{y})|^2
\dx.}
\label{eq:C--S sum 2}
\end{align}
Now, for any $k\ge1$ and any $\nu\in \Lambda_\ell=\{\nu\in\mathcal{F}\mid |\nu|=\ell,\ \norm{\nu}_\infty\le r\}$ with $\ell\le k-1$, let
$$
R_{\nu,\ell,k}:=
\{\mu\in\Lambda_k\mid \nu\in S_\mu \}
=
\{\mu\in\mathcal{F}\mid |\mu|=k,\ \norm{\mu}_\infty\le r,\ \mu\ge \nu,\ \mu\neq\nu \}.
$$
Then, for fixed $k\ge1$  we can write 
\begin{align}
\bigcup_{\mu\in\Lambda_k}
\bigcup_{\nu\in S_\mu}
(\mu,\nu)
=&
\bigcup_{\ell=0}^{k-1}
\bigcup_{\nu\in \Lambda_\ell}
\bigcup_{\mu\in R_{\nu,\ell,k}}
(\mu,\nu).
\end{align}
Thus, we have
\begin{align}
\sum_{\mu\in \Lambda_k}
\sum_{\nu\in S_\mu}
\epsilon(\mu,\nu)
a(\bs{y})
|\rho^\nu\nabla\partial^\nu u(\bs{y})|^2
=
\sum_{\ell=0}^{k-1}
\sum_{\nu\in \Lambda_\ell}
a(\bs{y})|\rho^\nu\nabla\partial^\nu u(\bs{y})|^2
\sum_{\mu\in R_{\nu,\ell,k}}
\epsilon(\mu,\nu).
\label{eq:change index}
\end{align}
Now, note that 
$k-\ell={\sum_{j\in\supp \mu}\mu_j -\sum_{j\in\supp \mu}\nu_j=} |\mu-\nu|$.
 Thus, we have $\frac{\mu!}{\nu!}\le r^{k-\ell}$. It follows that
\begin{align}
\sum_{\mu\in R_{\nu,\ell,k}}
\epsilon(\mu,\nu)
=&
\sum_{ \nu\in R_{\nu,\ell,k} }
\frac{\mu!}{\nu!}\frac{\rho^{\mu-\nu}|\psi|^{\mu-\nu}}{(\mu-\nu)!}
\le
r^{k-\ell}
\sum_{ \nu\in R_{\nu,\ell,k} }
\frac{\rho^{\mu-\nu}|\psi|^{\mu-\nu}}{(\mu-\nu)!} \\
&\le 
r^{k-\ell}
\sum_{|\tau|=k-\ell}
\frac{\rho^{\tau}|\psi|^{\tau}}{\tau!}
\le
r^{k-\ell}\frac1{(k-\ell)!}\kappa^{k-\ell}.
\label{eq:R bd}
\end{align}
Then, substituting \eqref{eq:R bd} into 
{\eqref{eq:change index} we obtain from 
\eqref{eq:C--S sum 2}}
%
\begin{align}
\sigma_k
\le 
\sum_{\ell=0}^{k-1}
\frac1{(k-\ell)!}
(r\kappa)^{k-\ell}
\sigma_\ell.
\end{align}
From the assumption we have $\kappa<{\frac{\ln 2}{r}}$. Thus, we can take $\delta=\delta(r)<1$ such that
$\kappa<\delta {\frac{\ln 2}r}.$

We show $\sigma_k\le \sigma_0\delta^k$ for all $k\ge 0$ by induction. This is clearly true for $k=0$. Suppose $\sigma_{\ell}\le \sigma_0\delta^{\ell}$ holds for $\ell=0,\dotsc,k-1$. Then, for $\ell=k$ we have
\begin{align}
\sigma_k
&\le
\sum_{\ell=0}^{k-1}
\frac1{(k-\ell)!}
(r\kappa)^{k-\ell}
\sigma_\ell
\le
\sum_{\ell=0}^{k-1}
\frac1{(k-\ell)!}
(r\kappa)^{k-\ell}
\sigma_0\delta^\ell
\le
\sum_{\ell=0}^{k-1}
\frac1{(k-\ell)!}
(\delta {\ln 2})^{k-\ell}
\sigma_0\delta^\ell\\
&=
\sigma_0\delta^k
\sum_{\ell=0}^{k-1}
\frac1{(k-\ell)!}
({\ln 2})^{k-\ell}
\le\sigma_0\delta^k(e^{{\ln 2}}-1)=\sigma_0\delta^k,
\end{align}
which completes the proof.
\end{proof}
With the notation
\begin{align}
\amin(\bs{y}):=\essinf_{x\in\dom} a(x,\bs{y}),\quad\text{ and }
\amax(\bs{y}):=\esssup_{x\in\dom} a(x,\bs{y}),
\end{align}
we have the following corollary, 
{where here and from now on we set $r=1$.}
\begin{cor}\label{cor:pd uy lognormal}
Suppose $(\psi_j)$ satisfies Assumption \ref{assump:Bprime} with a positive sequence $(\rho_j)$. 
Then, {for} $C_0=C_0(1)$ {as in Proposition \ref{prop:pd uy lognormal}} 
for any $\fraku\subset\bbN$ of finite cardinality we have 
\begin{align}
\norm{\mixedfirst{u(\bs{y})}{\bs{y}}}_{V}
\le 
\sqrt{C_0}
\frac{\norm{f}_{\Vdual}}{{\amin(\bs{y})}}
\prod_{j\in\fraku}\frac{1}{\rho_j}<\infty,
\quad\text{almost surely,}\label{eq:bd uy}
\end{align}
where $\norm{\cdot}_{\Vdual}$ is the norm in the dual space $\Vdual$. 
The same bound holds also for $\norm{\mixedfirst{u^s}{\bs{y}}}_{V}$, with $\bs{y}=(y_1,\dotsc,y_s,0,0,\dotsc)$.
\end{cor}
\begin{proof}
First, if $\bs{y}\in\bbR^\bbN$ satisfies 
$\|\sum_{j\ge1}y_j\psi_j\|_{L^{\infty}(\dom)}<\infty$, then 
we have $\frac{1}{(\amin(\bs{y}))}<\infty$:
\begin{align}
{\amin(\bs{y})}
&\ge
\big(\inf_{x\in\dom}a_0(x)\big)
\exp\Big(-\esssup_{x\in\dom}\Big|\sum_{j\ge1}y_j\psi_j(x)\Big|\Big)
,
\end{align}
and thus 
\begin{align}
\frac1{\amin(\bs{y})}
\le \frac1{\big(\inf_{x\in\dom}a_0(x)\big)}
\exp\Big(\esssup_{x\in\dom}\big|\sum_{j\ge1}y_j\psi_j(x)\big|\Big).\label{eq:bd aminy 1}
\end{align}
Now, from 
$(1/\rho_j)\in\ell^{q}$ for some $q\in(0,\infty)$, in view of \cite[Remark 2.2]{Bachmayr_etal_2016_ESAIM_part2} we have 
$$
\E
\bigg[
\exp\bigg(
k
\bigg\|
\sum_{j\ge1}y_j\psi_j
\bigg\|_{L^{\infty}(\dom)}
\bigg)
\bigg]<\infty,$$ for any $0\le k{<} \infty$. 
Thus, {$\|\sum_{j\ge1}y_j\psi_j\|_{L^{\infty}(\dom)}<\infty$, and} the right hand side of \eqref{eq:bd uy} is bounded with full (Gaussian) measure. 
We remark that the 
${\mathcal{B}(\bbR^{\bbN})}/\mathcal{B}(\bbR)$-measurability of the mapping 
$\bs{y}\mapsto \big\| \sum_{j\ge1}y_j\psi_j \big\|_{L^{\infty}(\dom)}$ 
is not an issue.
See \cite[Remark 2.2]{Bachmayr_etal_2016_ESAIM_part2} noting the continuity of norms, together with, for example, \cite[Appendix to IV. 5]{Reed.M_Simon_1980_book}. 

Now, {noting} that the standard argument regarding the continuous {dependence} of the solution of the variational problem \eqref{eq:wk formulation param} on $f$, we have $\int_{\dom} a(\bs{y})|\nabla(u(\bs{y}))|^2\dx\le \frac{\norm{f}_{\Vdual}^2}{\amin(\bs{y})}$. Then the claim follows from Proposition \ref{prop:pd uy lognormal}, noting that for any $\fraku\subset\bbN$ of finite cardinality we have
\begin{align}
\amin(\bs{y})\int_{\dom} 
\Big|\nabla\Big(\mixedfirst{u}{\bs{y}}\Big)\Big|^2
\dx
\le 
\sum_{\substack{\mu\in\mathcal{F}\\\norm{\mu}_\infty\le 1}}
{\rho^{2\mu}}
\int_{\dom} a(\bs{y})|\nabla(\partial^\mu u(\bs{y}))|^2\dx.
\end{align}
\end{proof}
\begin{rem}\label{rem:on bbdness}
We note that following a similar discussion {to the} above, $\amax(\bs{y})$ can be bounded almost surely. Thus, under the Assumption \ref{assump:Bprime}, the well-posedness of the problem \eqref{eq:wk formulation param} readily follows almost surely. 
{Further,  Assumption \ref{assump:Bprime} implies the measurability of the mapping $\bbR^s\ni\bs{y}\mapsto \G(u^s(\cdot,\bs{y}))\in \bbR$. 
See \cite[Corollary 2.1, Remark 2.2]{Bachmayr_etal_2016_ESAIM_part2} noting $\G\in\Vdual$, together with the fact that a strongly $\scrF$-measurable $V$-valued mapping is weakly $\scrF$-measurable. For more details on the measurability of vector-valued functions, see for example, \cite{Reed.M_Simon_1980_book,Yosida.K_book_1995reprint}.}
\bksq
\end{rem}
\section{QMC integration error with product weights}\label{sec:QMC error}
Based on the bound on mixed derivatives obtained in the previous section, now we derive a QMC convergence rate with product weights.

We first introduce some notations. 
Let
\begin{align}
\varsigma_j(\lambda)
:=
2
\left(
\frac{ \sqrt{2\pi}\exp(\alpha_j^2/\Lambda^{*}) }{ \pi^{2-2\Lambda^{*}}(1-\Lambda^{*})\Lambda^{*} }
\right)^\lambda
\zeta\left(\lambda+\frac12\right),\label{eq:varsigmaj}
\end{align}
where $\Lambda^{*}:=\frac{2\lambda-1}{4\lambda}$, and $\zeta(x):=\sum_{k=1}^\infty k^{-x}$ denotes the Riemann zeta function. 

We record the following result from Graham et al. \cite{GKNetal2014:lognormal}. 
\begin{theorem}\label{thm:Graham etal thm15}
\textup{(\cite[Theorem 15]{GKNetal2014:lognormal})}
Let $F\in \mathcal{W}^s$. Given $s$, $n\in\mathbb{N}$ with $2\le n\le 10^{30}$, weights $\bs{\gamma}=(\gamma_{\fraku})_{\fraku\subset\mathbb{N}}$, and the standard normal density function $\phi$, a randomly shifted lattice rule with $n$ points in $s$ dimensions can be constructed by a component-by-component algorithm such that, for all $\lambda\in(1/2,1]$,
\begin{align}
\sqrt{
\E^{\bs{\Delta}}
\big|
I_s(F)
-
\mathcal{Q}_{s,n}(\bs{\Delta};F)
\big|^2
}
\le
{9}
\Bigg(
\sum_{\emptyset\neq\fraku\subseteq\{1:s\}}
\gamma_{\fraku}^\lambda
\prod\limits_{j\in\fraku}\varsigma_j(\lambda)
\Bigg)^{\frac1{2\lambda}}
n^{-\frac{1}{2\lambda}}
\norm{F}_{\mathcal{W}^s}.
\end{align}
\bksq
\end{theorem}
For the weight function \eqref{eq:weight fct} we assume {that the} $\alpha_j$ satisfy for some constants $0<\alpha_{\min}<\alpha_{\max}<\infty${,}
\begin{align}
\max\Big\{\frac{\ln 2}{\rho_j},\alpha_{\min}\Big\}<\alpha_j\le \alpha_{\max},\qquad j\in\mathbb{N}.\label{eq:cond alphaj low up}
\end{align}
For example, {under Assumption \ref{assump:Bprime} letting} $\alpha_j:=1+\frac{\ln 2}{\rho_j}$ satisfies \eqref{eq:cond alphaj low up}
{with $\alpha_{\min}:=1$ and $\alpha_{\max}:=1+\sup_{j\ge1}\frac{\ln 2}{\rho_j}$.}

We have the following bound on $\norm{F}_{\mathcal{W}^s}^2$. The argument is essentially by Graham et al. \cite[Theorem 16]{GKNetal2014:lognormal}.
\begin{prop}\label{prop:Theorem 16'}
Suppose Assumption \ref{assump:Bprime} is satisfied with 
a positive sequence $(\rho_j)$ such that 
\begin{align}
(1/\rho_j)\in\ell^{1}.\label{eq:cond rho ell 1}
\end{align}
Then, we have
\begin{align}
\norm{F}_{\mathcal{W}^s}^2
\le
(C^*)^2\sumu\!
\frac1{\gamma_{\fraku}}
\bigg(\frac1{\prod\limits_{j\in\fraku}\rho_j}\bigg)^2
\prod\limits_{j\in\fraku}
\frac{1}{\alpha_j-({\ln 2})/\rho_j},
\label{eq:bd FWnorm}
\end{align}
with a positive constant 
$
C^*:=
\frac{\norm{f}_\Vdual\norm{\G}_\Vdual\sqrt{C_0}}{\inf_{x\in{\dom}}a_0(x)}
\left[\exp\left(
\frac12\sum_{j\ge1}\frac{(\ln2)^2}{\rho_j^2}+\frac2{\sqrt{2\pi}}\sum_{j\ge1}\frac{\ln2}{\rho_j}
\right)\right]<\infty.
$
\end{prop}
\begin{proof}
In this proof we abuse the notation slightly and $\bs{y}$ always denotes $(y_1,\dotsc,y_s,0,0,\dotsc)\in\mathbb{R}^{\mathbb{N}}$. 
From \eqref{eq:cond psi b1} and \eqref{eq:cond rho ell 1}, {in view of}
Corollary \ref{cor:pd uy lognormal} for $\Prob_Y$-almost every $\bs{y}$ we have 
\begin{align}
\left|
\mixedfirst{F}{\bs{y}}
\right|
\le \norm{\G}_{\Vdual}\norm{\mixedfirst{u^s}{\bs{y}}}_V
\le \norm{\G}_{\Vdual}\sqrt{{C_0}}\frac1{\prod\limits_{j\in\fraku}\rho_j}
\frac{\norm{f}_{\Vdual}}{\amin(\bs{y})}.\label{eq:bd F}
\end{align}
Since $$
\sup_{x\in\dom}\sum_{j\ge1}|y_j||\psi_j(x)|
\le 
\Big(
\sup_{j\ge1}\frac{|y_j|}{\rho_j}
\Big)
\sup_{x\in\dom}\sum_{j\ge1}\rho_j|\psi_j(x)|\le \Big(
\sum_{j\ge1}\frac{|y_j|}{\rho_j}
\Big)
\sup_{x\in\dom}\sum_{j\ge1}\rho_j|\psi_j(x)|,$$
the condition \eqref{eq:cond psi b1} 
and equations \eqref{eq:bd F} and \eqref{eq:bd aminy 1} together with $y_j=0$ {for $j>s$,} imply
\begin{align}
\left|
\mixedfirst{F}{\bs{y}}
\right|
\le 
\frac{K^*}{\prod\limits_{j\in\fraku}\rho_j}
\prod_{j\in\{1:s\}}\exp\bigg(
\frac{\ln2}{\rho_j}|y_j|
\bigg),
\end{align}
where $K^*:=\frac{\norm{f}_\Vdual\norm{\G}_\Vdual\sqrt{{C_0}}}
{\inf_{x\in{\dom}} a_0(x)}
$.
Then it follows {from \eqref{eq:F norm}} that
\begin{align}
\norm{F}&_{\mathcal{W}^s}^2
{=}
\sumu
\frac1{\gamma_{\fraku}}
\int_{\mathbb{R}^{|\fraku|}}\!\!\!
\left(
\int_{\mathbb{R}^{s-|\fraku|}}
\left|
\mixedfirst{F}{\bs{y}}(\bs{y}_\fraku;\bs{y}_{\{1:s\}\setminus\fraku})
\right|
\prod\limits_{j\in\{1:s\}\setminus\fraku}
\phi(y_j)\dbsy_{\{1:s\}\setminus\fraku}
\right)^2\!\!
\prod\limits_{j\in\fraku}
w_j^2(y_j)\dbsy_\fraku
\\
\le&
\sumu\!
\frac1{\gamma_{\fraku}}
\int_{\mathbb{R}^{|\fraku|}}\!\!\!
\left(
\int_{\mathbb{R}^{s-|\fraku|}}
\frac{K^*}{\prod\limits_{j\in\fraku}\rho_j}
\prod_{j\in\{1:s\}}\exp\bigg(
\frac{\ln2}{\rho_j}|y_j|
\bigg)
\prod\limits_{j\in\{1:s\}\setminus\fraku}
\phi(y_j)\dbsy_{\{1:s\}\setminus\fraku}
\right)^2\!\!
\prod\limits_{j\in\fraku}
w_j^2(y_j)\dbsy_\fraku\\
=&
(K^*)^2\sumu\!
\frac1{\gamma_{\fraku}}
\bigg(\frac1{\prod\limits_{j\in\fraku}\rho_j}\bigg)^2\nonumber\\
&\times
\left(
\int_{\mathbb{R}^{s-|\fraku|}}
\prod_{j\in\{1:s\}\setminus\fraku}\exp\bigg(
\frac{\ln2}{\rho_j}|y_j|
\bigg)
\prod\limits_{j\in\{1:s\}\setminus\fraku}
\phi(y_j)\dbsy_{\{1:s\}\setminus\fraku}
\right)^2\nonumber\\
&\times
\int_{\mathbb{R}^{|\fraku|}}
\prod_{j\in\fraku}\exp\bigg(
\frac{2\ln2}{\rho_j}|y_j|
\bigg)
\prod\limits_{j\in\fraku}
w_j^2(y_j)\dbsy_\fraku.
\end{align}
{Note that this takes essentially the same form as \cite[(4.14)]{GKNetal2014:lognormal}. Thus, the rest of the proof is in parallel to that of \cite[Theorem 16]{GKNetal2014:lognormal}.}
%

{Noting that }$2\alpha_j-\frac{2\ln2}{\rho_j}<0$, {and} following the same argument as {in }\cite[(4.15)--(4.17)]{GKNetal2014:lognormal}, we have
\begin{align}
\norm{F}_{\mathcal{W}^s}^2
\le&
(K^*)^2\sumu\!
\frac1{\gamma_{\fraku}}
\bigg(\frac1{\prod\limits_{j\in\fraku}\rho_j}\bigg)^2
\left(
\prod\limits_{j\in\{1:s\}\setminus\fraku}
2\exp\Big(\frac{(\ln2)^2}{2\rho_j^2}\Big)\Phi\Big(\frac{\ln2}{\rho_j}\Big)
\right)^2
\prod_{j\in\fraku}\frac1{\alpha_j-\frac{\ln2}{\rho_j}},
\end{align}
with $\Phi(\cdot)$ denoting the cumulative standard normal distribution function. 
Comparing this to \cite[Equation (4.17)]{GKNetal2014:lognormal}, the statement follows from the rest of the proof of \cite[Theorem 16]{GKNetal2014:lognormal}.
\end{proof}
%
{As in \cite[Theorem 17]{GKNetal2014:lognormal}, from Theorem \ref{thm:Graham etal thm15}} and Proposition \ref{prop:Theorem 16'} we have the following.
\begin{prop}\label{prop:Graham et al. Theorem 17'}
For each $j\ge 1$, let $w_j(t)=\exp(-2\alpha_j|t|)$ ($t\in\mathbb{R}$) with $\alpha_j$ satisfying \eqref{eq:cond alphaj low up}. Given $s$, $n\in\mathbb{N}$ with $2\le n\le 10^{30}$, weights $\bs{\gamma}=(\gamma_{\fraku})_{\fraku\subset\mathbb{N}}$, and the standard normal density function $\phi$, a randomly shifted lattice rule with $n$ points in $s$ dimensions can be constructed by a component-by-component algorithm such that, for all $\lambda\in(1/2,1]$,
\begin{align}
\sqrt{
\E^{\bs{\Delta}}
\left|
I_s(F)
-
\mathcal{Q}_{s,n}(\bs{\Delta};F)
\right|^2
}
\le
9C^*C_{\bs{\gamma},s}(\lambda)n^{-\frac{1}{2\lambda}},
\end{align}
with
\begin{align}
C_{\bs{\gamma},s}(\lambda):=
\left(
\sum_{\emptyset\neq\fraku\subseteq\{1:s\}}
\gamma_{\fraku}^\lambda
\prod\limits_{j\in\fraku}\varsigma_j(\lambda)
\right)^{\frac1{2\lambda}}
\left(
\sum_{\fraku\subseteq\{1:s\}}
\frac1{\gamma_{\fraku}}
\bigg(\frac1{\prod\limits_{j\in\fraku}\rho_j}\bigg)^2
\prod\limits_{j\in\fraku}
	\frac{1}{[\alpha_j-{\ln2}/{\rho_j}]}
\right)^\frac12,\label{eq:C_gammas}
\end{align}
and $C^*$ defined {as} in Proposition \ref{prop:Theorem 16'}.
\bksq
\end{prop}
We choose weights of the product form
\begin{align}
\gamma_\fraku=\gamma^*_\fraku(\lambda)
:=
\left[
\bigg(\frac1{\prod\limits_{j\in\fraku}\rho_j}\bigg)^2
\prod\limits_{j\in\fraku}
\frac{1}
{\varsigma_j(\lambda)[\alpha_j-\ln 2/\rho_j]}
\right]^{\frac{1}{1+\lambda}}\label{eq:weight}
\end{align}
Then, it turns out that under a suitable value of $\lambda$ the constant \eqref{eq:C_gammas} can be bounded independently of $s$, and we have the QMC {error bound} as follows.
\begin{theorem}\label{thm:conv rate}
For each $j\ge 1$, let $w_j(t)=\exp(-2\alpha_j|t|)$ ($t\in\mathbb{R}$) with $\alpha_j$ satisfying \eqref{eq:cond alphaj low up}.  Let $\varsigma_{\max}(\lambda)$ be $\varsigma_j$ defined by \eqref{eq:varsigmaj} but $\alpha_j$ being replaced by $\alpha_{\max}$. 
Suppose $(\psi_j)$ satisfies Assumption \ref{assump:B}. 
Suppose further that, we choose $\lambda$ as 
\begin{align}
\lambda=
\begin{cases}\label{eq:def lambda q}
\frac1{2-2\delta}\text{ for arbitrary }\ \delta\in(0,\frac12]
&\text{ when }q\in(0,\frac23]\\
\frac{q}{2-q}
&\text{ when }q\in(\frac23,1],
\end{cases}
\end{align}
{and} choose the weights {$\gamma_\fraku$} as {in} \eqref{eq:weight}. 
%

Then, given $s$, $n\in\mathbb{N}$ with $n\le 10^{30}$, 
and the standard normal density function $\phi$, a randomly shifted lattice rule with $n$ points in $s$ dimensions can be constructed by a component-by-component algorithm such that
\begin{align}
\sqrt{
\E^{\bs{\Delta}}
\left|
I_s(F)
-
\mathcal{Q}_{s,n}(\bs{\Delta};F)
\right|^2
}
\le
\begin{cases}
9C_{\rho,q,\delta}C^* n^{-(1-\delta)}
&\text{when } 0<q\le\frac23,\\
9C_{\rho,q}C^* n^{-\frac{2-q}{2q}}
&\text{when } \frac23<q\le1.
\end{cases}\label{eq:error order}
\end{align}
where the constant $C_{\rho,q,\delta}$, (resp. $C_{\rho,q}$) is independent of $s$ but depends on $\rho:=(\rho_j)$, $q$ and $\delta$ (resp. $\rho$ and $q$), and $C^*$ is defined {as} in Proposition \ref{prop:Theorem 16'}.

In particular, with $\alpha_j:=1+\ln2/\rho_j$ we have $\gamma_\fraku=
\Big[
\Big(\frac1{\prod\limits_{j\in\fraku}\rho_j}\Big)^2
\prod\limits_{j\in\fraku}
\frac{1}
{\varsigma_j(\lambda)}
\Big]^{\frac{1}{1+\lambda}}$, and the same result as above holds with the
finite constants $C_{\rho,q,\delta}$, and $C_{\rho,q}$ {both} given by
$$
{C_{\rho,q,\delta}=C_{\rho,q}=}
\Bigg(
	\prod_{j=1}^{\infty}
	\bigg(
	1+
	\Big(
		\frac{\varsigma_j(\lambda)}{\rho_j^{2\lambda}}
	\Big)^{\frac{1}{1+\lambda}}
	\bigg)
	-1
\Bigg)^{\frac{1}{2\lambda}}
\Bigg(
	\prod_{j=1}^{\infty}
	\bigg(
	1+
	\Big(
		\frac{\varsigma_j(\lambda)}{\rho_j^{2\lambda}}
	\Big)^{\frac{1}{1+\lambda}}
	\bigg)
\Bigg)^{\frac{1}{2}},
$$
with $\lambda$ given by \eqref{eq:def lambda q}.
\end{theorem}
\begin{proof}
Let $\beta_j(\lambda):=\Big(
	\frac{(\varsigma_j(\lambda))^{\frac1{\lambda}}}
	{\rho_j^2[\alpha_j-\ln 2/\rho_j]}\Big)
	^{\frac{\lambda}{1+\lambda}}$. Observe that with the choice of weights \eqref{eq:weight} we have
\begin{align}
C_{\bs{\gamma},s}(\lambda)
&=
\Bigg(
	\sum\limits_{\emptyset\neq\fraku\subseteq\{1:s\}}
	\prod\limits_{j\in\fraku}
	\beta_j(\lambda)
\Bigg)^{\frac1{2\lambda}}
\Bigg(
		\sum\limits_{\fraku\subseteq\{1:s\}}
		\prod\limits_{j\in\fraku}
		\beta_j(\lambda)
\Bigg)^\frac12 \\
&=
\left(
	\bigg( \prod\limits_{j=1}^{s} (1+\beta_j(\lambda)) \bigg)
	-1
\right)^{\frac1{2\lambda}}
\left(
	\prod\limits_{j=1}^{s} (1+\beta_j(\lambda))
\right)^\frac12.
\end{align}
Now, let $\mathcal{J}:=\inf_{j\ge 1}(\alpha_j-\ln2/\rho_j)$, which is a positive value from \eqref{eq:cond alphaj low up}. 
Further, note that $\varsigma_j(\lambda)\le \varsigma_{\max}(\lambda)$ for $j\ge1$. 
Then, from $\beta_j(\lambda)\ge0$ we have
\begin{align}
\prod\limits_{j=1}^{s} (1+\beta_j(\lambda))
\le 
\prod\limits_{j=1}^{s} \exp({\beta_j(\lambda)})
\le \exp\Big(\sum_{j\ge1}{\beta_j(\lambda)}\Big)
\le 
\exp\left(
\left[
\frac{[\varsigma_{\max}(\lambda)]^{\frac{1}{\lambda}}}{\mathcal{J}}
\right]^{\frac{\lambda}{1+\lambda}}
\sum_{j\ge1}
\left[
\frac{1}{\rho_j}
\right]^{\frac{2\lambda}{1+\lambda}}
\right).
\end{align}
Thus, if 
$
\sum_{j\ge1}
\left[
\frac{1}{\rho_j}
\right]^{\frac{2\lambda}{1+\lambda}}
<\infty
$
 we can conclude that $C_{\bs{\gamma},s}(\lambda)$ {is} bounded independently of $s$. 

We discuss the relation between $q$ and the exponent ${\frac{2\lambda}{1+\lambda}}$. 
 First note that from $\lambda\in(\frac12,1]$, we have
$
\frac23<{\frac{2\lambda}{1+\lambda}}\le 1.
$
Suppose $0<q\le\frac23$.
in this case, we always have $q<{\frac{2\lambda}{1+\lambda}}$, and thus $(1/\rho_j)\in\ell^{\frac{2\lambda}{1+\lambda}}$. Thus, {$\sum_{j\ge1}
\left[
\frac{1}{\rho_j}
\right]^{\frac{2\lambda}{1+\lambda}}
<\infty$} follows. Letting 
$
\lambda:=\frac1{2-2\delta}
$ 
with an arbitrary $\delta\in(0,\frac12]$, we obtain the result for $q\in(0,\frac23]$. Next, consider the case $\frac23<q\le1$. Then, letting $\lambda:=\lambda(q)=\frac{q}{2-q}$, we have $\lambda\in(1/2,1]$ and 
\begin{align}
\frac{2\lambda}{1+\lambda}
=
\frac{2\frac{q}{2-q}}{1+\frac{q}{2-q}}
=
\frac{2{q}}{2-q+q}=q,
\end{align}
and thus $\sum_{j\ge1}
\left[
\frac{1}{\rho_j}
\right]^{\frac{2\lambda}{1+\lambda}}
<\infty$.
\end{proof}
\section{{Application to a wavelet stochastic model}}\label{sec:applications}

%
%
Cioica et al. \cite{Cioica.P.A_etal_2012_BIT} considered a stochastic model {in} which users can choose the smoothness at {will}. In this section, we consider the Gaussian case, and show that the theory developed in Section \ref{sec:QMC error} can be applicable for the model with a wide range of smoothness. 
\subsection{Stochastic model}
For simplicity we assume $\dom$ is a bounded convex polygonal domain. Consider a wavelet system $(\varphi_{\xi})_{\xi\in\nabla}$ that is a Riesz basis for $L^2(\dom)$-space. We explain the notations and outline the standard properties we assume as follows. The indices $\xi\in\nabla$ typically encodes both the scale, often denoted by $|\xi|$, and the spatial location, and also the type of the wavelet. Since our analysis does not rely on the choice of a type of wavelet, we often use the notation $\xi=(\ell,k)$, and 
$\nabla=\{(\ell,k)\mid \ell\ge \ell_0,k\in\nabla_{\ell}\}$ where $\nabla_{\ell}$ is some countable index set. The scale level $\ell$ of $\varphi_\xi $ is denoted by $|\xi|=|(\ell,k)|=\ell$. 
Furthermore, $(\widetilde{\varphi}_{\xi})_{\xi\in\nabla}$ denotes the dual wavelet basis, i.e., $\langle {\varphi}_{\xi}, \widetilde{\varphi}_{\xi'} \rangle_{L^2(\dom)}=\delta_{\xi\xi'}$, $\xi,\xi'\in\nabla$.

In the following, $\alpha \lesssim \beta$ means that $\alpha$ can be bounded by some constant times $\beta$ uniformly with respect to any parameters on which $\alpha$ and $\beta$ may depend. Further, $\alpha\sim \beta$ means that $\alpha \lesssim \beta$ and $\beta \lesssim \alpha$.

We list the assumption on wavelets: 
\begin{enumerate}[(W1)]
\item \label{eq:item wavelets basis}
the wavelets $(\varphi_{\xi})_{\xi\in\nabla}$ form a Riesz basis for $L^2(\dom)$;
\item the cardinality of the index set ${\nabla_{\ell}}$ {satisfies} $\#{\nabla_{\ell}}=C_{\nabla} 2^{\ell d}$ for some constant $C_{\nabla}>0$;
\item 
the wavelets are local. That is, the supports of $\varphi_{{\ell,k}}$ are contained in balls of diameter $\sim {2^{-\ell}}$, and do not overlap too much in the following sense: there exists a constant $M>0$ independent of $\ell$ such that for each given $\ell$ for any $x\in\dom$, 
\begin{align}
\#\{
k\in\nabla_{\ell}
\mid 
\varphi_{\ell,k} (x)\neq 0
\}
\le M;\label{eq:wavelet level-wise finite overlap}
\end{align}
\item the wavelets satisfy the cancellation property 
\begin{align*}
|\langle v,\varphi_{\xi}\rangle_{L^2(\dom)}|
\lesssim 2^{-|\xi|(\frac{d}{2}+\tilde{m})}|v|_{W^{\tilde{m},\infty}(\supp(\varphi_{\xi}))},
\end{align*}
for $|\xi|\ge \ell_0$ with some parameter $\tilde{m}\in\bbN$, where $|\cdot|_{W^{\tilde{m},\infty}}$ denotes the usual Sobolev semi-norm. 
That is, the inner product is small when the function $v$ is smooth on the support $\supp(\varphi_{\xi})$;
\item\label{eq:item wavelets Besov characterisation} the wavelet basis induces characterisations of Besov spaces  $B^t_{\Besovq}(L_{\Besovp}(\dom))$ for ${1 \le}\Besovp,\Besovq<\infty$ and all $t$ with $d\max\{1/\Besovp-1,0\}<t<t_*$ for some parameter $t_*>0$. The upper bound $t_*$ depends on the choice of wavelet basis. Since $t$ we consider is typically small, here for simplicity we may \textit{define} the Besov norm as 
\begin{align}
\norm{v}_{B^t_{\Besovq}(L_{\Besovp}(\dom))}
:=
\Bigg(
\sum_{\ell=\ell_0}^{\infty}
2^{\ell\big(
t+d\big(\frac12-\frac{1}{\Besovp}\big)
\big)\Besovq}
\bigg(
\sum_{k\in\nabla_{\ell}}
|\langle v, \tilde{\varphi}_{\ell,k}\rangle_{L^2(\dom)}|^{\Besovp}
\bigg)^{\frac{\Besovq}{\Besovp}}
\Bigg)^{\frac1{\Besovq}},
\end{align}
\item\label{eq:item wavelets growth} the wavelets satisfy 
\begin{align}
\sup_{x\in\dom}|\varphi_{\ell,k}(x)|=C_{\varphi} 2^{\frac{\beta_0 d}{2}\ell}
\qquad
\text{ with some }{\beta_0}\in\mathbb{R_{+}},\label{eq:Klaus wavelet decay}
\end{align}
for some constant $C_{\varphi}>0$. 
Typically we have $\varphi_{\ell,k}\sim 2^{\frac{d}2\ell}\psi(2^{\ell}(x-x_{\ell,k}))$, for some bounded function $\psi$. In this case we have $\beta_0=1.$
\end{enumerate}
See \cite[section 2.1]{Cioica.P.A_etal_2012_BIT} and references therein for further details. See also \cite{Cohen.A_book_wavelet_2003,DeVore.R.A_1998_Acta_nonlinear,Urban.K_2002_book_wavelets_simulation}.

We now investigate a stochastic model expanded by the wavelet basis described above. 
Let $\{Y_{\ell,k}\}$ be a collection of independent standard normal random variables on a suitable probability space $(\Omega',\scrF',\Prob')$. 
We assume the random field \eqref{eq:random coeff} is given with $T$ such that 
\begin{align}
T(x,\omega')
=
\sum_{\ell=\ell_0}^\infty\sum_{k\in\nabla_{\ell}}
Y_{\ell,k}(\omega')
\sigma_{\ell} \varphi_{\ell,k}(x),\label{eq:Klaus model}
\end{align}
where 
\begin{align}
\sigma_{\ell}:=2^{-\frac{\beta_1d }2\ell}\text{ with }\beta_1>1.
\label{eq:cond wavelet sigma}
\end{align}
{
From $\E_{\Prob'}
\Big(
\sum_{\ell=\ell_0}^{\infty} 
\sum_{k\in\nabla_{\ell}} Y_{\ell,k}(\omega')^2 \sigma_{\ell}^2\Big)
=C_{\nabla}\sum_{\ell=\ell_0}^{\infty} 2^{-(\beta_1-1)d \ell}
<\infty$, in view of (W\ref{eq:item wavelets basis}) the series \eqref{eq:Klaus model} converges $\Prob'$-almost surely in $L^2(\dom)$.}

{To replace \eqref{eq:random coeff},} we consider the following log-normal stochastic model:
 \begin{align}
a(x,\omega')=a_*(x)
+
{a_0(x)}
\exp
\bigg(
\sum_{\ell=\ell_0}^\infty
\sum_{k\in\nabla_{\ell}}
Y_{\ell,k}(\omega')\sigma_{\ell} \varphi_{\ell,k}(x)\bigg).\label{eq:sto model}
\end{align}
In the following, we argue that we can reorder $\sigma_{\ell} \varphi_{\ell,k}$ lexicographically as $\sigma_{j} \varphi_{j}$ and see it as $\psi_j$, while keeping the law. 

Throughout this section, we assume {that} the parameters $\beta_0$ {and} $\beta_1$ {satisfy} \begin{align}
{0}<{\beta_1} - {\beta_0},
\end{align}
and {that} point evaluation $\varphi_{\ell,k}(x)$ ($(\ell,k)\in\nabla$) is well-defined for any $x\in\dom$. Under this assumption, reordering $(Y_{\ell,k}\sigma_{\ell} \varphi_{\ell,k})$ lexicographically does not change the law of \eqref{eq:Klaus model} on $\bbR^{\dom}$. To see this, from the Gaussianity it suffices to show that the covariance function $\E_{\Prob'}[T(\cdot)T(\cdot)]\from\dom\times\dom\to\bbR$ is invariant under the reordering. 

Fix $x\in\dom$ arbitrarily. For any $L$, $L'$ ($L{>} L'$), from the independence of $\{Y_{\ell,k}\}$ we have 
\begin{align}
\E_{\Prob'}
\bigg(
\sum_{\ell=\ell_0}^{L} 
\sum_{k\in\nabla_{\ell}} Y_{\ell,k}(\omega') \sigma_{\ell} \varphi_{\ell,k}(x)
-
\sum_{\ell=\ell_0}^{L'} 
\sum_{k\in\nabla_{\ell}} Y_{\ell,k}(\omega') \sigma_{\ell} \varphi_{\ell,k}(x)
\bigg)^2&=
\sum_{\ell=L'+1}^L \sum_{k\in\nabla_{\ell}}
{\sigma_{\ell}^2} \varphi_{\ell,k}^2(x)\\
&\le 
{C_{\varphi}^2M
\sum_{\ell=L'+1}^L 2^{-(\beta_1-\beta_0 )d\ell}}<\infty.
\end{align}
Hence, the sequence $\big\{\sum_{\ell=\ell_0}^{L} 
\sum_{k\in\nabla_{\ell}} Y_{\ell,k}(\omega') \sigma_{\ell} \varphi_{\ell,k}(x)\big\}_L$ is convergent in $L^2(\Omega',\Prob')$. The continuity of the inner product $\E_{\Prob'}[\cdot,\cdot]$ on $L^2(\Omega')$ in each variable yields
\begin{align}
\E_{\Prob'}[T(x_1)T(x_2)]
&=
\sum_{\ell=\ell_0}^{\infty} 
\sum_{k\in\nabla_{\ell}}
\sum_{\ell'=\ell'_0}^{\infty} 
\sum_{k'\in\nabla_{\ell'}}
\E_{\Prob'}[
Y_{\ell,k}(\omega') \sigma_{\ell} \varphi_{\ell,k}(x_1)
Y_{\ell',k'}(\omega') \sigma_{\ell'} \varphi_{\ell',k'}(x_2)
]\\
&=\sum_{\ell=\ell_0}^{\infty} 
\sum_{k\in\nabla_{\ell}} \sigma_{\ell}^2
\varphi_{\ell,k}(x_1)\varphi_{\ell,k}(x_2),
\qquad \text{for any $x_1$, $x_2\in\dom$.}
\end{align} 
But we have
 $\sum_{\ell=\ell_0}^{\infty} 
 	\sum_{k\in\nabla_{\ell}} \sigma_{\ell}^2
 	|\varphi_{\ell,k}(x_1)\varphi_{\ell,k}(x_2)|
\le{C_{\varphi}^2M
\sum_{\ell=L'+1}^L 2^{-(\beta_1-\beta_0 )d\ell}}
$. Hence,
$$
\E_{\Prob'}[T(x_1)T(x_2)]=\sum_{j\ge1}\sigma_{j}^2\varphi_{j}(x_1)\varphi_{j}(x_2),
\qquad x_1,\,x_2\in \dom.
$$
Following {a} similar discussion, we see that the series $\sum_{j\ge1}\sigma_{j}^2 y_j\varphi_{j}(x)$ converges in $L^2(\Omega)$ for each $x\in\dom$, and has the covariance function $\sum_{\ell=\ell_0}^{\infty} 
\sum_{k\in\nabla_{\ell}} \sigma_{\ell}^2
\varphi_{\ell,k}(x_1)\varphi_{\ell,k}(x_2)$. Hence the law on $\bbR^{\dom}$ is the same. 
Thus, abusing the notation slightly we write $T(\cdot,\bs{y}):=T(\cdot,\omega')$, $y_{\ell,k}:=Y_{\ell,k}(\omega')$, $\Omega=\bbR^{\bbN}:=\Omega'$, $\scrF:=\scrF'$, $\Prob_{Y}:=\Prob'$, and $\E[\cdot]:=\E_{\Prob'}[\cdot]$.

Next, we discuss the applicability of the theory developed in Section \ref{sec:QMC error} to the wavelet stochastic model above.
We need to check Assumption \ref{assump:B}. 
%

Take $\theta\in(0,\frac{d}2({\beta_1}-{\beta_0}))$, 
and 
for $\xi=(\ell,k)$ let 
\begin{align}
\rho_{\xi}:=c2^{\theta |\xi|}=c2^{\theta \ell},\label{eq:def rho}
\end{align}
with some constant
 $0<c<{\ln 2}\big(M{C_{\varphi}}\sum_{\ell=\ell_0}^\infty 2^{\ell(\theta-\frac{d}2({\beta_1}-{\beta_0}))}\big)^{-1}$.

{Then,} by virtue of the locality property \eqref{eq:wavelet level-wise finite overlap} we have \eqref{eq:cond psi b1} as follows{:}
\begin{align}
\sup_{x\in \dom}\sum_{\xi}
\rho_{\xi}
|\sigma_{\xi}\varphi_{\xi}(x)|
\le
\sum_{\ell=\ell_0}^\infty
\rho_{\ell}
\sup_{x\in\dom}
\sum_{k\in\nabla_{\ell}}
|2^{-\frac{\beta_1d\ell}2}\varphi_{\ell,k}(x)|
\le&
cM{C_{\varphi}}\sum_{\ell=\ell_0}^\infty
2^{\theta \ell}
2^{-\frac{\beta_1d\ell}2}
2^{\frac{\beta_0 d}{2}\ell}<\ln2. 
\label{eq:psi rho cond Klaus}
\end{align}

Further, note that by reordering for sufficiently large $j$ we have 
\footnote{
To see this, first recall that there are $\mathcal{O}(2^{{\ell d}})$ wavelets at level $\ell$. 
Thus, for an arbitrary but 
sufficiently large $j$ we have 
$$
2^{\ell_jd}\lesssim j \lesssim 2^{(\ell_j+1)d}.
$$
for some $\ell_j$.

Let $\xi_j\in \nabla_{\ell_j}$ be the index corresponding to $j$. 
Since $|\xi_j| = \ell_j$, we have
\begin{align}
\sup_{x\in D} |\sigma_{j} \varphi_{j} (x)| 
= \sup_{x\in D} |\sigma_{\ell_j} \varphi_{\xi_j} (x)|
{=}
 C_{\varphi} 2^{-\frac{\beta_1 d}2 \ell_j} 2^{\frac{\beta_0 d}{2} \ell_j}
\lesssim C_{\varphi}{2^{\frac{d}2\beta^*}} j^{-\frac{\beta_1}2 +\frac{\beta_0}2}
{,\text{ for any }\beta^*>\beta_1-\beta_0.}
\end{align}
The opposite direction can be derived as, from $\beta_1-\beta_0>0$, 
\begin{align}
j^{-\frac{\beta_1}2+\frac{\beta_0}{2}}
\lesssim 
2^{{-\ell_j d(\frac12(\beta_1-\beta_0))}}
=\frac{1}{C_{\varphi}}
\sup_{x\in D}|\sigma_{j} \varphi_{j}(x)|.
\end{align}
The relation $\rho_j\sim j^{\frac{\theta}{d}}$ can be checked similarly.} 
\begin{align}
\sup_{x\in\dom}|\sigma_{j}\varphi_{j}(x)|
\sim
j^
{
-\frac12({\beta_1} - {\beta_0 }),
}\label{eq:Klaus reorder decay}
\end{align}
and
\begin{align}
\rho_{j}\sim j^{\frac{\theta}{d}}.
\end{align}
Thus, to have $\sum_{j\ge 1} \frac1{\rho_j}<\infty$, the weakest condition on the summability on $(1/\rho_j)$ for Assumption \ref{assump:B} to be satisfied, it is necessary (and sufficient) to
have $\theta > {d}$. 

The following proposition summarises the discussion above.
\begin{theorem}\label{thm:wavelet QMC error}
Suppose the random coefficient \eqref{eq:random coeff} is given by $T$ as in \eqref{eq:Klaus model} with $(\varphi_{\ell,k})$ that satisfies \eqref{eq:Klaus wavelet decay}, and non-negative numbers $(\sigma_ {\ell})$ that satisfy \eqref{eq:cond wavelet sigma}. Let $(\rho_{\xi})$ be defined by \eqref{eq:def rho}. Further, assume ${\beta_0}$ and $\beta_1$ satisfy
\begin{align}
\frac{2}{q} < {\beta_1} - {\beta_0 },\label{eq:cond beta01}
\end{align}
for some ${q} \in (0,1]$. 
Then, the reordered system $(\sigma_{j}\varphi_{j})$ with the reordered $(\rho_j)$ 
satisfies Assumption \ref{assump:B}, 
 and {under the same conditions on $w_j(t)$, $\alpha_j$, and $\varsigma_j$ as in Theorem \ref{thm:conv rate}} we have the QMC error {bound} \eqref{eq:error order} with this $q$.
\end{theorem}
\begin{proof}
{Take $\theta\in(\frac{d}q,\frac{d}2(\beta_1-\beta_0))$, and}
 define 
$(\rho_{\xi})$ as {in} \eqref{eq:def rho}, reorder the components lexicographically, and
denote {the reordered $(\rho_{\xi})$} by $(\rho_j)$. Then, we have \eqref{eq:cond 1/rho b2}
\begin{align}
\sum_{j\ge 1} \bigg( \frac1{\rho_j}\bigg)^{q}
\lesssim 
\sum_{j\ge 1} \bigg( \frac1{j}\bigg)^{ \frac{{q}\theta}d }<\infty.
\end{align}
Further, from ${\theta-\frac{\beta_1d}{2}+\frac{\beta_0 d}{2}}<0$ we have \eqref{eq:psi rho cond Klaus}, and thus \eqref{eq:cond psi b1} holds.  
Hence, from the discussion in this section Assumption \ref{assump:B} is  satisfied, and thus {in view of Theorem \ref{thm:conv rate}} we have  \eqref{eq:error order}.
\end{proof}
\subsection{H\"{o}lder smoothness of the realisations}\label{sec:Holder smoothness}
Often, random fields $T$ with realisations that are not smooth are regularly of interest. 
In this section, we see that the stochastic model we consider \eqref{eq:sto model} allows reasonably rough random fields (H\"{o}lder smoothness) for $d=1,2$. The result is shown via Sobolev embedding results. We provide a necessary and sufficient condition to have {specified} Sobolev smoothness (Theorem \ref{thm:Besov smoothness}). Recall that embedding results are in general optimal (see, for example, \cite[4.12, 4.40--4.44]{Adams.R_book_Sobolev}), and in this sense, we have a sharp condition for our model to have H\"{o}lder smoothness. 
A building block is a Besov characterisation of the realisations which is essentially
 {due to} Cioica et al. \cite[Theorem 6]{Cioica.P.A_etal_2012_BIT}. Here we {define} $s:=s(L):=\sum_{\ell=\ell_0}^{L}\#(\nabla_{\ell})$, that is, the truncation is considered in terms of the level $L$.
\begin{theorem}\label{thm:Besov smoothness}\textup{(\cite[Theorem 6]{Cioica.P.A_etal_2012_BIT})}
{Let $\Besovp,\Besovq\in[1,\infty)$, and }$t\in(d\max\{1/\Besovp-1,0\},t_*)$, where $t_*$ is the parameter in {(W\ref{eq:item wavelets Besov characterisation})}. Then, 
\begin{align}
t< d\Big(
\frac{\beta_1 - 1}2
\Big)\label{eq:cond t Besov}
\end{align}
if and only if $T\in B^t_{\Besovq}(L_{\Besovp}(\dom))$ a.s.
Further, if \eqref{eq:cond t Besov} is satisfied, then the stochastic model \eqref{eq:sto model} satisfies 
$\E[\norm{T^{s(L)}}_{B^t_{\Besovq}(L_{\Besovp}(\dom))}^{\Besovq}]\le \E[\norm{T}_{B^t_{\Besovq}(L_{\Besovp}(\dom))}^{\Besovq}]<\infty$ for all $L\in\bbN$.
\end{theorem}
\begin{proof}
First, 
from the proof of \cite[Theorem 6]{Cioica.P.A_etal_2012_BIT}, 
we see that $T\in B^t_{\Besovq}(L_{\Besovp}(\dom))$ a.s., is equivalent to 
$$
\sum_{\ell=\ell_0}^{\infty}
2^{\ell(t+d(1/2-1/{\Besovp}) )\Besovq}
\sigma_{\ell}^{\Besovq}
(
\# \nabla_{\ell}
)^{\Besovq/\Besovp}
\sim 
\sum_{\ell=\ell_0}^{\infty}
2^{\ell\Besovq
(t-\frac{d}2(\beta_1-1) )}<\infty,
$$
which holds from the assumption $t< d\big(
\frac{\beta_1 - 1}2
\big)$. Similarly, from the proof of \cite[Theorem 6]{Cioica.P.A_etal_2012_BIT} we have 
$$
\E[\norm{T}_{B^t_{\Besovq}(L_{\Besovp}(\dom))}^{\Besovq}]{\lesssim}
\sum_{\ell=\ell_0}^{\infty}
2^{\ell(t+d(1/2-1/{\Besovp}) )\Besovq}
\sigma_{\ell}^{\Besovq}
(
\# \nabla_{\ell}
)^{\Besovq/\Besovp}<\infty.
$$
Finally, {from (W\ref{eq:item wavelets Besov characterisation}) we have } $\E[\norm{T^{s}}_{B^t_{\Besovq}(L_{\Besovp}(\dom))}^{\Besovq}]
{=
\sum_{\ell=\ell_0}^{s}
2^{\ell(t+d(1/2-1/{\Besovp}) )\Besovq}
\E\big[ \big(	\sum_{k\in\nabla_{\ell}}|Y_{\ell,k}|^\Besovp	\big)^{\Besovq/\Besovp} \big]}
\le \E[\norm{T}_{B^t_{\Besovq}(L_{\Besovp}(\dom))}^{\Besovq}]$, {completing} the proof.
\end{proof}
To establish the H\"{o}lder smoothness, we employ embedding results. To invoke {them}, we {first} establish {that} the realisations {are} continuous; we want the measurability, and want to keep the law of $T$ on $\bbR^{\dom}$.

The H\"{o}lder norm involves taking {the} supremum over the uncountable set $\dom$, and thus {whether the} resulting function $\Omega\ni \bs{y}\mapsto \norm{T(\cdot,\bs{y})}_{C^{t_1}(\bddom)}\in\bbR${, where $t_1\in(0,1]$ is a H\"{o}lder exponent,} is an $\bbR$-valued random variable {is not immediately clear}. We see that by the continuity the measurability is preserved. 

Sobolev embeddings are achieved by finding a suitable representative by changing values of functions on measure zero sets of $\dom$. This change could affect the law on $\bbR^{\dom}$, since it is determined by the laws of arbitrary finitely many random variables $(T(x_1),\dotsc,T(x_m))$ ($\{x_i\}_{i=1,\dotsc,m}\subset\dom$) on $\bbR^{m}$. To avoid this, we establish the existence of continuous modification, thereby taking the continuous element of a Besov function that respects the law of $T$ from the outset. 

We make an assumption on the covariance function so that realisations of $T$ have continuous paths. We assume there exist positive constants $\iota_1$, $C_{\mathrm{KT}}$, and $\iota_2(>d)$ satisfying 
\begin{align}
\E[|T(x_1)-T(x_2)|^{\iota_1}]\le 
C_{\mathrm{KT}}\norm{ x_1 - x_2 }_2^{\iota_2},
\qquad \text{ for any }x_1,x_2\in \dom.\label{eq:cond KT}
\end{align}
Then, by virtue of Kolmogorov--Totoki's theorem \cite[Theorem 4.1]{Kunita.H_2004_SDE_Levy_diffeomorphisms} $T$ has a continuous modification. Further, the continuous modification is uniformly continuous on $\dom$ and it can be extended to the closure $\bddom$.

A H\"{o}lder smoothness of $(\varphi_{k,l})$ is sufficient for \eqref{eq:cond KT} to hold.
\begin{prop}
Suppose that $(\sigma_{\ell})$ satisfies \eqref{eq:cond wavelet sigma}. 
Further, suppose that for each $(\ell,k)\in\nabla$, the function $\varphi_{\ell,k}$ is 
$t_0$-H\"{o}lder continuous {on $\dom$ for some  $t_0\in(0,1]$}. Then, $T$ has a modification that is uniformly continuous on $\dom$ and can be extended to the closure $\bddom$.
\end{prop}
\begin{proof}
{It suffices to show \eqref{eq:cond KT} holds.} 
Fix $x_1,x_2\in \dom$ arbitrarily. First note that 
\begin{align}
{\sigma_*^2}:=\E[|T(x_1)-T(x_2)|^{2}]
&=\sum_{\ell=\ell_0}^{\infty} \sum_{k\in\nabla_{\ell}} 
\sigma_{\ell}^2( \varphi_{\ell,k}(x_1) - \varphi_{\ell,k}(x_2))^2\\
&{\le C} \norm{x_1 - x_2}_2^{2t_0} \sum_{\ell=\ell_0}^{\infty} \sum_{k\in\nabla_{\ell}} 
\sigma_{\ell}^2<\infty,
\end{align} 
{where $C$ is the $t_0$-H\"{o}lder constant.} 
Then, since $T(x_1)-T(x_2)\sim \mathcal{N}(0,{\sigma_*^2})$ we observe that, with $X_{\mathrm{std}}\sim\mathcal{N}(0,1)$ we have
\begin{align}
\E[|T(x_1)-T(x_2)|^{2m}]
&=\E[|X_{\mathrm{std}} \sigma_*|^{2m}]
=\sigma_*^{2m}  \E[|X_{\mathrm{std}}|^{2m}] \\
&{\le C^m}\norm{x_1 - x_2}_2^{2t_0 m} 
\bigg(
\sum_{\ell=\ell_0}^{\infty} \sum_{k\in\nabla_{\ell}} \sigma_{\ell}^2
\bigg)^{{m}}
\E[|X_{\mathrm{std}}|^{2m}],
\text{ for any }m\in\bbN.
\end{align}
Taking $m>\frac{d}{2t_0}$, we have \eqref{eq:cond KT} 
{with $\iota_1:=2m$, 
$C_{\mathrm{KT}}:= C^m
\big(
\sum_{\ell=\ell_0}^{\infty} \sum_{k\in\nabla_{\ell}} \sigma_{\ell}^2
\big)^{m}
\E[|X_{\mathrm{std}}|^{2m}]$, and $\iota_2:=2t_0m(>d)$}
, and thus the statement follows. 
\end{proof}
In the following, we assume $\varphi_{\ell,k}$ is $t_0$-H\"{o}lder continuous on $\dom$ {for some $t_0\in(0,1]$}. Note that under this assumption, we may assume $\varphi_{\ell,k}$ is continuous on $\bddom$.

Using {the fact} that $T(\cdot,\bs{y})\in B^t_{2}(L_{2}(\dom))=H^{t}(\dom)$ a.s., 
now we establish expected the H\"{o}lder smoothness of the random coefficients $a$. From this result, for example, the convergence rate of the finite element method using the piecewise linear functions are readily obtained. 

{First, we argue that to analyse the H\"{o}lder smoothness of the realisations of $a$, without loss of generality we may assume $a_*\equiv0$ and $a_0\equiv1$.} {To see this, suppose $a_*$, $a_0$ in \eqref{eq:sto model} satisfies $a_*,a_0\in C^{t_1}(\bddom)$ for some $t_1\in(0,1]$.} By virtue of
\begin{align}
|\e^{a} - \e^{b}|=\bigg|\int_{a}^{b}\e^{r}\dr \bigg|
\le \max\{\e^{a},\e^{b}\}|b-a|\le (\e^{a}+\e^{b})|b-a|
\text{ for all }a,b\in\bbR,
\label{eq:exp bd}
\end{align}
for any $x_0,x_1,x_2\in\bddom$ ($x_1\neq x_2$) we have
\begin{align}
{\big|\e^{T(x_0)}\big|+\frac{\big|\e^{T(x_1)}-\e^{T(x_2)}\big|}{\norm{x_1-x_2}_2^{t_1}}
\le 
\Big(\sup_{x\in\bddom}\big|\e^{T(x)}\big|\Big)
\bigg(1+2
\frac{|T(x_2) - T(x_3)|}{\norm{x_1-x_2}_2^{t_1}}
\bigg).}
\end{align}
{Noting that 
$\norm{a_0 \e^{T}}_{C^{t_1}(\bddom)}\le C_{t_1} 
\norm{a_0}_{C^{t_1}(\bddom)}\norm{\e^{T}}_{C^{t_1}(\bddom)}
$ (see, for example \cite[p. 53]{Gilbarg.D_T_book_2001}) we have}
\begin{align}
\norm{a}_{C^{t_1}(\bddom)}
&\le{
\norm{a_*}_{C^{t_1}(\bddom)}
+
C_{t_1}\norm{a_0}_{C^{t_1}(\bddom)}
\bigg(\sup_{x\in\bddom}|\e^{T(x)}|\bigg)}
\bigg(1+2
\norm{T}_{C^{t_1}(\bddom)}
\bigg).
\end{align}
{Thus, given $a_*,a_0\in C^{t_1}(\bddom)$, it suffices to show $\bigg(\sup_{x\in\bddom}|\e^{T(x)}|\bigg)
\bigg(1+2
\norm{T}_{C^{t_1}(\bddom)}
\bigg)<\infty$ for the H\"{o}lder smoothness of the realisations of $a$. Therefore, in the rest of this subsection, for simplicity we assume $a_*\equiv0$ and $a_0\equiv1$.}

In order to invoke embedding results we assume $t_*$ satisfies $\frac{d}2<\lfloor t_*\rfloor$, {and that} we can take $t\in(0,\frac{d}2(\beta_1-1))$ such that $\frac{d}2<\lfloor t\rfloor$.
For {the latter} to hold, taking $\beta_1\ge3$, implying $\frac{d}2<\lfloor \frac{d}2(\beta_1-1)\rfloor$, is sufficient, which is always satisfied for the presented QMC theory to be applicable. See Remark \ref{rem:regularity allowed}. 

{Now, take $t_1\in(0,1]\cap(0,\lfloor t\rfloor-\frac{d}2]$. Then, from $B^t_{2}(L_{2}(\dom))=H^{t}(\dom)$ and the Sobolev embedding (for example, \cite[Theorem 4.12]{Adams.R_book_Sobolev}) we have}
\begin{align}
\norm{a}_{C^{t_1}(\bddom)}&\lesssim 
\Big(\sup_{x\in\bddom}|a(x)|\Big)
\bigg(1+2
\norm{T}_{B^t_{2}(L_{2}(\dom))}
\bigg),
\label{eq:Holder bd a}
\end{align}
Similarly, we have 
$\norm{a^{s}}_{C^{t_1}(\bddom)}
\lesssim
\big(\sup_{x\in\bddom}|a^s(x)|\big)
\Big(1+2
\norm{T^s}_{B^t_{2}(L_{2}(\dom))}
\Big)$. 

We want to take the expectation of $\norm{a}_{C^{t_1}(\bddom)}$. To do this, we establish the $\scrF/\mathcal{B}(\bbR)$-measurability of $\bs{y}\mapsto\norm{a(\cdot,\bs{y})}_{C^{t_1}(\bddom)}$. Taking continuous modifications of $T$ if necessary, we may assume paths of $a$ are continuous on $\bddom$. 
Then, from the continuity of the mapping
$$
\{(x_1,x_2)\in\bddom\times\bddom\mid x_1\neq x_2 \}\ni (x_1,x_2)\mapsto
\frac{|a(x_1)-a(x_2)|}{\norm{x_1-x_2}_2^{t_1}}\in\bbR,
$$
with a countable set 
$G$ that is dense in $\{(x_1,x_2)\in\bddom\times\bddom\mid x_1\neq x_2 \}\subset\bbR^d\times\bbR^d$ we have
\begin{align}
\sup_{x_1,x_2\in \bddom,\,x_1\neq x_2}
\frac{|a(x_1)-a(x_2)|}{\norm{x_1-x_2}_2^{t_1}}
=
\sup_{(x_1,x_2)\in G}
\frac{|a(x_1)-a(x_2)|}{\norm{x_1-x_2}_2^{t_1}}.
\end{align}
Thus, $\bs{y}\mapsto \norm{a(\cdot,\bs{y})}_{C^{t_1}(\bddom)}$, and by the same argument, $\bs{y}\mapsto \norm{a^s(\cdot,\bs{y})}_{C^{t_1}(\bddom)}$, are $\mathcal{B}(\bbR^\bbN)/\mathcal{B}(\overline{\bbR})$-measurable, where $\overline{\bbR}:=\bbR\cup\{-\infty\}\cup\{\infty\}$.

From 
$\E[\norm{T^s}_{C(\bddom)}]\lesssim
\E[\norm{T^s}_{ B^t_{2}(L_{2}(\dom)) }]\le 
\E[\norm{T}_{ B^t_{2}(L_{2}(\dom)) }]
\lesssim
{\big(\sum_{\ell=\ell_0}^{\infty}
2^{\ell
(2t-{d}(\beta_1-1) )}
\big)^{1/2}}<\infty$ independently of $s$, and
$\E[\norm{T}_{C(\bddom)}]\lesssim{\big(\sum_{\ell=\ell_0}^{\infty}
2^{\ell
(2t-d(\beta_1-1) )}
\big)^{1/2}}<\infty$, following the discussion by Charrier \cite[Proof of Proposition 3.10]{Charrier.J_2012_SINUM_strong_weak} utilising the Fernique's theorem there exists a constant $M_p>0$ independent of $p$ such that
\begin{align}
\E[\exp(p\norm{T^s(\cdot,\bs{y})}_{C(\bddom)})]
\Big\},
\E[\exp(p\norm{T(\cdot,\bs{y})}_{C(\bddom)})]
\Big\}<M_{p},
\end{align}
for any $p\in (0,\infty)$. Together with, $\sup_{x\in\bddom}|a(x)|\le \exp(\sup_{x\in\bddom}|T(x)|)$, we have 
$$
\E[(\sup_{x\in\bddom}|a^s(x)|)^{ 2p }],
\E[(\sup_{x\in\bddom}|a(x)|)^{ 2p }]<M_{2p},
\quad \text{for any $p\in(0,\infty)$.}$$
Hence, from \eqref{eq:Holder bd a} we conclude that
\begin{align}
\E[\norm{a}_{C^{t_1}(\bddom)}^{ p }]
&\le
\max\{1,2^{p-1/2}\}\sqrt{\E\Big[
\Big(\sup_{x\in\bddom}|a(x)|\Big)^{ 2p }
\Big]}
\sqrt{1+4^{ p }
\E\big[
\norm{T}_{B^t_{2}(L_{2}(\dom))}^{2p}
\big]} <\infty.
\end{align}
Similarly, we have
$$\E[\norm{a^s}_{C^{t_1}(\bddom)}^{ p }]
\le
\max\{1,2^{p-1/2}\}\sqrt{\E\Big[
\Big(\sup_{x\in\bddom}|a^s(x)|\Big)^{ 2p }
\Big]}
\sqrt{1+4^{ p }
\E\big[
\norm{T^s}_{B^t_{2}(L_{2}(\dom))}^{2p}
\big]}<\infty,
$$
where the right hand side can be bounded independently of $s$.
\begin{rem}\label{rem:regularity allowed}
We provide a remark regarding the smoothness of the realisations that the currently developed theory permits. From the conditions imposed on the basis functions, e.g., the summability conditions, random fields with smooth realisations are easily in the scope of the QMC theory applied to PDEs. Here, the capability of taking reasonably rough random field into account is of interest.
Typically, $L^2$ wavelet Riesz basis have growth rate $\beta_0=1$. 
Then, the condition $2 < \beta_1   -  \beta_0 ${, the weakest condition on $\beta_1$ in Theorem \ref{thm:wavelet QMC error},} is equivalent to 
\begin{align}
\beta_1=3+\ep,\quad\text{ for any }\ep>0.
\tag{A1}\label{eq:cond beta}
\end{align}
In view of Theorem \ref{thm:Besov smoothness}, the smaller the decay rate $\beta_1$ of $\sigma_\ell$ is, the rougher the realisations are. We discuss the smoothness of the realisations achieved by $\beta_1=3+\ep$ for some small $\ep>0$, one of the values of $\beta_1$ as small as possible. In applications, $d=1,2,3$ are of interest. See Table \ref{table:smoothness}, which summarises the condition \eqref{eq:cond t Besov} with \eqref{eq:cond beta}.

From $B^t_{2}(L^{2}(\dom))=H^{t}(\dom)$, in view of Theorem \ref{thm:Besov smoothness},  
$T(\cdot, \bs{y})\in H^{t}(\dom)$ a.s. if and only if the condition \eqref{eq:cond t Besov}, holds. 
We recall the following embedding results. See, for example, \cite[p. 85]{Adams.R_book_Sobolev}.
\begin{table}
\centering
\begin{tabular}{|c|c|c|}
\hline 
 & $t<\frac{d}2(\beta_1 - 1)$ & $t<\frac{d}2(\beta_1 - 1)$ with $\beta_1=3+\ep$ for some ($\ep>0$) \tabularnewline
\hline 
\hline 
$d=1$ & $t<(\beta_1 - 1)/2$ & $t<1 + \ep$\tabularnewline
\hline 
$d=2$ & $t<(\beta_1 - 1)$ & $t<2 + \ep$\tabularnewline
\hline 
$d=3$ & $t<\frac{3}2(\beta_1 - 1)$ & $t<3+\ep$\tabularnewline
\hline 
\end{tabular}
\caption{Range of the exponent $t$ for realisations of $T$ to have $H^{t}$-smoothness
and the smallest bound on $t$ allowed by the presented QMC theory when $\beta_0 =1$}
\label{table:smoothness}
\end{table}
For $d=1,2$, and $3$ respectively, with $\beta_1=3+\ep$ the condition \eqref{eq:cond t Besov} reads $t<1 + \ep$, $t<2 + \ep$, and 
$t<3 + \ep${, where we rescaled $\ep$ depending on $d$.}

For $d = 1, 2,$ this seems to be rough enough. 
For $d = 1$, $H^1(\dom)$ is characterised as a space of absolutely continuous functions.  Since in practice we employ a suitable numerical method to solve PDEs, the validity of point evaluations demands $a(\cdot,\bs{y})\in  C(\dom)$. 
For $d = 2$, we know $H^2(\dom)$ can be embedded to $C^{0,t}(\bddom)$, ($t \in  (0,1)$). This is a standard
assumption to have the convergence of FEM with the hat function elements on polygonal domains.

For $d = 3$, we know $H^3(\dom) = H^{1+2}(\dom)$ can be embedded to $C^{1,t}(\bddom)$, ($t\in  (0,2 - \frac32] = 
(0, \frac12]$). 
In practice, we employ quadrature 
rules to compute the integrals in the bilinear form. That $a \in  C^{1,t}(\bddom)$ ($t \in(0, {\frac12}]$) is a reasonable 
assumption to get the convergence rate for FEM with quadratures. As a matter of fact, we want $a(\cdot,\bs{y}) \in  C^{2r}(\bddom)$
 to have the $\mathcal{O}(H^{2r})$ convergence of the expected $L^{p}(\Omega)$-moment of $L^2(\dom)$-error even for $C^2$-bounded domains. See 
\cite[Remark 3.14]{Charrier.J_Scheichl_Teckentrup_2013_FE_multilevel}, and
\cite[Remark 3.2]{Teckentrup.A.L_etal_2013_further_analysis}. 

Finally, we note these embedding results are in general optimal (see, for example, \cite[4.12, 4.40--4.44]{Adams.R_book_Sobolev}), and in this sense, together with the characterisation (Theorem \ref{thm:Besov smoothness}), the condition for our model to have H\"{o}lder smoothness is sharp.
\bksq
\end{rem}
\subsection{Dimension truncation error}
In this section we estimate the truncation error 
$
\E \norm{ u - u^s}_V
$. Again, the truncation is considered in terms of the level $L$ and we let $s=s(L)=\sum_{\ell=\ell_0}^{L}\#(\nabla_{\ell})$. Let $a^s$ be $a(x,\bs{y})$ with $y_j=0$ for $j>s$, and define $\amin^s(\bs{y})$, $\amax^s(\bs{y})$ accordingly. By a variant of Strang's lemma, we have
\begin{align}
\norm{ u - u^s}_V
\le \norm{a-a^s}_{L^\infty(D)}\frac{\norm{f}_{\Vdual}}{\amin(\bs{y})\amin^s(\bs{y})}
\label{eq:Strang}
\end{align}
for $\bs{y}$ such that $\amin(\bs{y})$, $\amin^s(\bs{y})>0$. This motivates us to derive an estimate on $\norm{a-a^s}_{L^\infty(D)}$.

Assuming a differentiability and a further summability {of $(\psi_j)$}, Charrier \cite{Charrier.J_2012_SINUM_strong_weak} obtained estimates on the moments of $\norm{a-a^s}_{C(\bddom)}$ and thus $\norm{ u - u^s}_V$, by the inequality of the same form {as} \eqref{eq:Strang}. 
See \cite[Proposition 3.4 and Theorem 4.2, together with Assumption 3.1]{Charrier.J_2012_SINUM_strong_weak}. A similar argument is employed in \cite{GKNetal2014:lognormal}. Bearing in mind the argument by Charrier uses the Fernique's theorem for separable Banach spaces \cite[Theorem 2.2, Proposition 2.3]{Charrier.J_2012_SINUM_strong_weak}, the same argument is applicable here by replacing $L^\infty(D)$ with $C(\bddom)$, which can be done following the discussion in Section \ref{sec:Holder smoothness}.

In the present paper, however, we impose no further smoothness condition of the wavelet basis functions. 
{We note that from \eqref{eq:Klaus reorder decay} and \eqref{eq:cond beta01}, we have $\sum_{j\ge1}\sup_{x\in\dom}|\sigma_j\varphi_{j}|^p<\infty$ for some $p\in(0,1]$. Thus, the theory developed by Graham et al. \cite{GKNetal2014:lognormal} can be applied to the scaled wavelet basis $\varphi_{\ell,k}\sigma_{\ell}$, which in turn, together with the truncation error estimate we obtain in the following, shows that in the theory developed in \cite{GKNetal2014:lognormal}, the assumption \cite[Assumption A2 (b)]{GKNetal2014:lognormal} that is used to obtain a truncation error estimate \cite[Theorem 8]{GKNetal2014:lognormal} is in general, in particular, for a wide class of wavelets basis, is not necessary.}
\begin{prop}
Let $u$ be the solution of the variational problem \eqref{eq:wk formulation param} with the coefficient given by the stochastic model \eqref{eq:sto model} defined with \eqref{eq:Klaus model} and \eqref{eq:cond wavelet sigma}. Let $u^{s(L)}$ be the solution of the same problem but with $y_j:=0$ for {$j>s(L)$}. 
Suppose $t\in(0,t_*)$, where $t_*$ is the parameter in (W{\ref{eq:item wavelets Besov characterisation}}), satisfies
$
t< d\big(
\frac{\beta_1 - 1}2
\big)$.
Then, we have 
\begin{align}
\E[\big\| u - u^{s(L)}\big\|_V]
\lesssim
\Big(\sum_{\ell=L+1}^{\infty}
2^{\ell
(2t-d(\beta_1-1) )}\Big)^{\frac12}.
\end{align}
\end{prop}
\begin{proof}
{For $t\in(0,\frac{d}2(\beta_1-1))$, choose $\Besovp_0\in[1,\infty)$ such that $\frac{d}{\Besovp_0}\le t$ so that we can invoke the Besov embedding results.} 
Since $\max\{d(\frac{1}{\Besovp_0}-1),0\}<t$, from Theorem \ref{thm:Besov smoothness} there exists a set $\Omega_0\subset\Omega$ such that $\Prob(\Omega_0)=1$ and $T(\cdot,\bs{y})\in B^t_{\Besovq}(L^{\Besovp_0}(\dom))$ for all $\bs{y}\in\Omega_0$ with any {$\Besovq\in[1,\infty)$}. 
Then, letting $T^L(x,\bs{y}):=\sum_{\ell=\ell_0}^{L} 
\sum_{k\in\nabla_{\ell}}  y_{\ell,k}  \sigma_{\ell} \varphi_{\ell,k}(x)$, from the embedding result of Besov spaces (\cite[Chapter 7]{Adams.R_book_Sobolev}), and the characterisation by wavelets (W\ref{eq:item wavelets Besov characterisation}) for any $L$, $L'\ge1$ ($L\ge L'$) we have
\begin{align}
\norm{T^L(\cdot,\bs{y})-T^{L'}(\cdot,\bs{y})}_{L^{\infty}(\dom)}
&\lesssim
\norm{T^L(\cdot,\bs{y})-T^{L'}(\cdot,\bs{y})}_{B^t_{\Besovq}(L^{\Besovp_0}(\dom))} \\
&\sim
\bigg(
\sum_{\ell=L'+1}^L 
2^{\ell(t+d(1/2-1/{\Besovp_0})) \Besovq }
\bigg(
\sum_{k\in\nabla_{\ell}}
|\sigma_{\ell}  y_{\ell,k} |^{\Besovp_0}
\bigg)^{\Besovq/\Besovp_0}
\bigg)^{1/\Besovq}<\infty,
\end{align}
for all $\bs{y}\in\Omega_0$. Thus, the sequence $\{T^{L}(\cdot,\bs{y})\}_L$ ($\bs{y}\in\Omega_0$) is Cauchy, and thus convergent in $L^{\infty}(\dom)$. 
Hence, we obtain 
\begin{align}
\norm{T(\cdot,\bs{y})-T^{L}(\cdot,\bs{y})}_{L^{\infty}(\dom)}^{\Besovq}
\lesssim
\sum_{\ell=L+1}^{\infty}
2^{\ell(t+d(1/2-1/{\Besovp}) )\Besovq}
\bigg(
\sum_{k\in\nabla_{\ell}}
|\sigma_{\ell}  y_{\ell,k} |^{\Besovp}
\bigg)^{\Besovq/\Besovp}\qquad \text{a.s.,}
\end{align}
for all $\Besovp\in[1,\infty)$ such that $\frac{d}{\Besovp}\le t$. 
For such {$\Besovp$ and any $\Besovq\in[1,\infty)$}, from \cite[Proof of Theorem 6]{Cioica.P.A_etal_2012_BIT}, we have
\begin{align}
\E\big[
\norm{T(\cdot,\bs{y})-T^{L}(\cdot,\bs{y})}_{L^{\infty}(\dom)}^{\Besovq}
\big]\lesssim
\sum_{\ell=L+1}^{\infty}
2^{\ell(t+d(1/2-1/{\Besovp}) )\Besovq}
\sigma_{\ell}^{\Besovq}
(
\# \nabla_{\ell}
)^{\Besovq/\Besovp}
\sim 
\sum_{\ell=L+1}^{\infty}
2^{\ell\Besovq
(t-\frac{d}2(\beta_1-1) )}<\infty.
\label{eq:T trunc bd}
\end{align}
Further, {from \eqref{eq:exp bd}} we have
\begin{align}
&\E\Big[\norm{a(x,\bs{y})-a^{s(L)}(x,\bs{y})}_{L^{\infty}(\dom)}^2\Big]\notag\\
&\le 
(\sup_{x\in\dom}|a_0(x)|^2)
\E[\exp(2\norm{T(\cdot,\bs{y})}_{L^{\infty}(\dom)})
+
\exp(2\|T^L(\cdot,\bs{y})\|_{L^{\infty}(\dom)})]
\E\big[\norm{T-T^L}_{L^{\infty}(\dom)}^2\big].
\label{eq:a trunc bd}
\end{align}
The sequence $(\rho_\xi)$ defined by \eqref{eq:def rho}, when reordered, satisfies {$(1/\rho_j)\in\ell^{\frac{d}{\theta}+\ep}$ for any $\ep>0$}. Thus, from the proof of Corollary \ref{cor:pd uy lognormal}, {as in} \cite[Remark 2.2]{Bachmayr_etal_2016_ESAIM_part2}, we have 
\begin{align}
\max\Big\{
\E[\exp(2\norm{T(\cdot,\bs{y})}_{L^{\infty}(\dom)})],
\E[
\exp(2\|T^L(\cdot,\bs{y})\|_{L^{\infty}(\dom)})]
\Big\}<M_2,\end{align} where the constant $M_2>0$ is independent of $L$.

Together with \eqref{eq:Strang}, we have
\begin{align}
\E[\norm{ u - u^s}_V]
\le 
\norm{f}_{\Vdual}
\E\Big[
\frac{1}{(\amin(\bs{y}))^4}\Big]^{\frac14}
\E\Big[
\frac{1}{(\amin^s(\bs{y}))^4}\Big]^{\frac14}
\E[\norm{a-a^s}_{L^\infty(D)}^2]^{\frac12}<\infty,
\label{eq:utrunc reduced to atrunc}
\end{align}
where Cauchy--Schwarz inequality is employed in the right hand side of \eqref{eq:Strang}. To see the finiteness of the right hand side of \eqref{eq:utrunc reduced to atrunc}, note that  
$$\frac1{\amin(\bs{y})}\le \frac1{\inf_{x\in\dom} a_0(x)}\exp(\|T\|_{L^{\infty}(D)}),\ \frac1{\amin^s(\bs{y})}\le\frac1{\inf_{x\in\dom} a_0(x)} \exp(\|T^L\|_{L^{\infty}(D)}),$$ and further, 
from the same argument as above, we have
\begin{align}
\max\Big\{
\E[\exp(4\norm{T(\cdot,\bs{y})}_{L^{\infty}(\dom)})],
\E[
\exp(4\|T^L(\cdot,\bs{y})\|_{L^{\infty}(\dom)})]
\Big\}<M_4,
\end{align}
where the constant $M_4>0$ is independent of $L$.

Therefore, from \eqref{eq:T trunc bd}, \eqref{eq:a trunc bd}, and \eqref{eq:utrunc reduced to atrunc} we obtain
\begin{align}
\E[\big\| u - u^{s(L)}\big\|_V]
\lesssim
\E\big[\norm{T-T^L}_{L^{\infty}(\dom)}^2\big]^{\frac12}
\lesssim
\Big(\sum_{\ell=L+1}^{\infty}
2^{\ell
(2t-d(\beta_1-1) )}\Big)^{\frac12}.
\end{align}
\end{proof}
We conclude this section with a remark on other examples to which the currently developed QMC theory is applicable. Bachmayr et al. \cite{Bachmayr_etal_2016_ESAIM_part2} considered so-called functions $(\scbasis_j)$ with finitely overlapping supports, for example, indicator functions of a partition of the domain $\dom$. It is easy to find a positive sequence $(\rho_j)$ such that Assumption \ref{assump:B} holds, and thus Theorem \ref{thm:conv rate} readily follows. However, for these examples, due to the lack of smoothness it does not seem that it is easy to obtain a meaningful analysis as given above, and thus we forgo elaborating them.
\section{Concluding remark}\label{sec:conclusion}
We considered a QMC theory for a class of elliptic partial differential equations with a log-normal random coefficient. Using an estimate on the partial derivative with respect to the parameter {$y_{\fraku}$} that is of product form, we established a convergence rate $\approx1$ of randomly shifted lattice {rules}. Further, we considered a stochastic model with wavelets, and analysed the smoothness of the realisations, and truncation errors. 
\section*{Acknowledgement} I would like to express my sincere gratitude to Frances Y. Kuo, Klaus Ritter, and Ian H. Sloan for their stimulating comments.

\printbibliography
\Addresses
\end{document}

%% file: my_arxiv_header.tex
\usepackage{amsmath} 
\usepackage{amsthm} 
\usepackage{amssymb}	
\usepackage{multicol} 
\usepackage{soul} 
\usepackage{cancel}
\usepackage{dsfont,bbm}
\usepackage{empheq,cases}
\usepackage[export]{adjustbox}

\DeclareFontFamily{OT1}{pzc}{}
\DeclareFontShape{OT1}{pzc}{m}{it}{<-> s * [1.10] pzcmi7t}{}
\DeclareMathAlphabet{\mathpzc}{OT1}{pzc}{m}{it}

\usepackage{calligra}
\DeclareMathAlphabet{\mathcalligra}{T1}{calligra}{m}{n}

\usepackage[low-sup]{subdepth}
\usepackage{hyperref}

\providecommand{\keywords}[1]{\textbf{\textit{Keywords: }} #1}

\renewcommand{\div}[1]{{\nabla} \cdot #1} 
\let\baraccent=\= 
\renewcommand{\=}[1]{\stackrel{#1}{=}} 

\newcommand{\norm}[1]{\left\| #1 \right\|}

\newcommand{\inv}{^{-1}}

\newcommand{\ep}{\varepsilon}

\newcommand{\bs}[1]{\ensuremath{\boldsymbol{#1}}}

\usepackage{graphicx}
\usepackage{amsmath,amsthm,amssymb,bm,color,comment}
\usepackage{mathrsfs,enumerate}
\usepackage{subfig}
\usepackage{caption}
\usepackage[font=small,format=plain,labelfont=bf,textfont=it]{caption}
\usepackage{url}
\numberwithin{equation}{section}
\definecolor{refkey}{rgb}{0.9451,0.2706,0.4941}
\definecolor{labelkey}{rgb}{0.9451,0.2706,0.4941}
\newtheorem{theorem}{Theorem}[section]
\newtheorem*{theorem*}{Theorem}
\newtheorem{cor}[theorem]{Corollary}
\newtheorem{lem}[theorem]{Lemma}

\newtheorem{prop}[theorem]{Proposition}
\theoremstyle{definition}

\newtheoremstyle{myremstyle}
  {\topsep} 
  {1.2\topsep} 
  {} 
  {} 
  {\itshape} 
  {.} 
  {.5em} 
  {} 

\theoremstyle{myremstyle} \newtheorem{rem}{Remark}

\newcommand\bksq{$~\hbox{\vrule width 2.5pt depth 2.5 pt height 3.5 pt}$}

\usepackage{fancybox,here,mathtools}

\newcommand*{\from}{\colon}
\newcommand*{\bigmid}{\:\big|\:}
\newcommand*{\Bigmid}{\:\Big|\:}

\newcommand{\bbR}{{\mathbb{R}}}
\newcommand{\bbN}{{\mathbb{N}}}

\newcommand{\scrF}{{\mathscr{F}}}
\newcommand{\fraku}{{\mathfrak{u}}}

\newcommand{\dom}{D}
\newcommand{\bddom}{\overline{\dom}}

\DeclareMathOperator*{\esssup}{{ess}\sup}
\DeclareMathOperator*{\essinf}{{ess}\inf}
\DeclareMathOperator{\supp}{supp}

\newcommand{\G}{\mathcal{G}}

\newcommand{\Prob}{\mathbb{P}}
\newcommand{\e}{\mathrm{e}}

\newcommand{\E}{\mathbb{E}}

\newcommand{\dx}{\,{\rm d}x}
\newcommand{\dy}{\,{\rm d}y}
\newcommand{\dbsy}{\,{\rm d}\boldsymbol{y}}

\newcommand{\dr}{\ensuremath{\,{\rm d}r}}

\usepackage[
	style=numeric,
	backend=bibtex,
	backref=true,
	sorting=nyt,
	maxbibnames=99,
	firstinits=true,
	doi=false,
	isbn=false
	]{biblatex}

\DefineBibliographyStrings{english}{%
  backrefpage = {page},
  backrefpages = {pages},
}

\renewbibmacro{in:}{}

\DeclareFieldFormat{postnote}{#1}
\DeclareFieldFormat{multipostnote}{#1}
	
\AtEveryBibitem{
\clearfield{month}
\clearfield{number}
\clearlist{location}
}
\AtEveryBibitem{\clearlist{language}} 
\DeclareFieldFormat[article]{title}{#1}
\AtEveryBibitem{%
  \ifentrytype{book}{%
    \clearfield{pages}%
  }{%
  }%
}
\AtEveryBibitem{%
\ifentrytype{book}{
    \clearfield{url}%
}{}
\ifentrytype{article}{
    \clearfield{url}%
}{}
\ifentrytype{incollection}{
    \clearfield{url}%
}{}
}
